\documentclass{article} 

\usepackage{amsthm,amsfonts,amsmath,amscd,amssymb,graphicx,tikz-cd,mathtools,mathrsfs,authblk,url,upgreek,comment}  

\newtheorem{theorem}{Theorem}
\newtheorem{lemma}[theorem]{Lemma}
\newtheorem{proposition}[theorem]{Proposition}

\newtheorem{definition}[theorem]{Definition}
\newtheorem{corollary}[theorem]{Corollary}
\newtheorem{example}[theorem]{Example}

\numberwithin{theorem}{section}

\newtheorem{conjecture}[theorem]{Conjecture}

\makeatletter
\newtheoremstyle{indented}
  {10pt}
  {10pt}
  {\addtolength{\@totalleftmargin}{1.5em}
   \addtolength{\linewidth}{-1.5em}
   \parshape 1 1.5em \linewidth}
  {}
  {\bfseries}
  {.}
  {.5em}
  {}
\makeatother

\newtheorem{remark}[theorem]{Remark}

\def\R{{\mathbb R}} 
 
\def\N{{\mathbb N}} 
\def\C{{\mathbb C}}

\def\gog{{\mathfrak g}} \def\tot{{\mathfrak t}}
\def\mult{{\mathrm{mult}}}
\def\CJ{{\mathcal J}}
\def\CD{{\mathcal D}}

\newcommand\xqed[1]{%
  \leavevmode\unskip\penalty9999 \hbox{}\nobreak\hfill
  \quad\hbox{#1}}
\newcommand\remdone{\xqed{$\triangle$}}

\begin{document}

\author{Colin McSwiggen}
\date{}

\title{Box splines, tensor product multiplicities \\ and the volume function}

\maketitle

\begin{abstract}
We study the relationship between the tensor product multiplicities of a compact semisimple Lie algebra $\gog$ and a special function $\CJ$ associated \mbox{to $\gog$}, called the volume function.  The volume function arises in connection with the randomized Horn's problem in random matrix theory and has a related significance in symplectic geometry.  Building on box spline deconvolution formulae of Dahmen--Micchelli and De Concini--Procesi--Vergne, we develop new techniques for computing the multiplicities from $\CJ$, answering a question posed by Coquereaux and Zuber.  In particular, we derive an explicit algebraic formula for a large class of Littlewood--Richardson coefficients in terms of $\CJ$.   We also give analogous results for weight multiplicities, and we show a number of further identities relating the tensor product multiplicities, the volume function and the box spline.  To illustrate these ideas, we give new proofs of some known theorems.
\end{abstract}

\tableofcontents

\section{Introduction}

An important combinatorial problem in representation theory is the determination of tensor product multiplicities.  Given two irreducible representations $V_\lambda, V_\mu$ of a compact semisimple Lie algebra $\gog$, or equivalently of the connected, simply connected group $G$ with Lie algebra $\gog$, we would like to compute the decomposition
$$V_\lambda \otimes V_\mu = \bigoplus_\nu C_{\lambda \mu}^\nu V_\nu,$$
where $\lambda, \mu, \nu$ are the corresponding highest weights.  Since characters combine multiplicatively under the tensor product and additively under the direct sum, this is equivalent to computing the structure constants of the algebra generated by the irreducible characters of $G$:
$$\chi_\lambda \chi_\mu = \sum_\nu C_{\lambda \mu}^\nu \chi_\nu.$$
Much is known about this problem from a combinatorial perspective, including algorithms for computing the multiplicities, although in general the problem is \texttt{\#P}-complete \cite{HN}.  The most widely studied case is $\gog = \mathfrak{su}(n)$, where the multiplicities are usually called Littlewood--Richardson coefficients as they are described by the famous Littlewood--Richardson rule \cite{LR}.

In this paper we study the relationship between the multiplicities $C_{\lambda \mu}^\nu$ and a special function $\CJ$ associated to $\gog$, called the volume function (see Definition \ref{def:J-def} below), which takes three arguments in a Cartan subalgebra $\tot \subset \gog$:
$$\CJ(\alpha, \beta ; \gamma), \qquad \alpha, \beta, \gamma \in \tot.$$
We develop multiple methods for computing $C_{\lambda \mu}^\nu$ from $\CJ$, and we demonstrate how each of these two objects can be used to study the other, leading to new results as well as new proofs of known theorems.  

The volume function is of independent significance in symplectic geometry, as well as in random matrix theory, where it is closely related to the joint spectral density for the randomized Horn's problem; see (\ref{eqn:p-from-J}) below.  All of the main results involving $\CJ$ in this paper can be reformulated in terms of this probability density, so that in a loose sense this paper can be interpreted as establishing an equivalence between two problems: the problem of computing tensor product multiplicities for a compact semisimple Lie algebra $\gog$, and the randomized Horn's problem for coadjoint orbits in $\gog^*$.  It is well known that these problems are related asymptotically via a semiclassical limit (discussed below in section \ref{sec:semiclassical}), but here we derive exact non-asymptotic formulae for $C_{\lambda \mu}^\nu$ in terms of finitely many values of $\CJ$.

The original motivation for this study was a question posed by Coquereaux and Zuber in \cite[sect. 2.3]{CZ1}, where they showed that certain values of the volume function for $\mathfrak{su}(n)$ can be expressed algebraically in terms of Littlewood--Richardson coefficients.  The result is a pair of identities that they called the $\CJ$-LR relations (see (\ref{eqn:JLR1}) below), which were later extended to arbitrary $\gog$ in \cite{CMZ}.  Coquereaux and Zuber asked whether these relations might be inverted, yielding an expression for $C_{\lambda \mu}^\nu$ in terms of $\CJ$.  We will answer this question in the affirmative, although the formulae that we obtain will be more or less explicit depending on $\gog$ and on the highest weights $(\lambda, \mu, \nu)$.

A concrete illustration of the relationship between $C_{\lambda \mu}^\nu$ and $\CJ$ comes from a construction of Berenstein and Zelevinsky \cite{BZ}, who defined a polytope $H_{\lambda \mu}^\nu$ such that $C_{\lambda \mu}^\nu$ is equal to the number of points in $H_{\lambda \mu}^\nu$ with integer coordinates.  It was shown in \cite{CMZ} that $\CJ(\lambda, \mu ; \nu)$ equals the volume of $H_{\lambda \mu}^\nu$.  Recovering $C_{\lambda \mu}^\nu$ from $\CJ$ therefore amounts to computing the number of integer points in $H_{\lambda \mu}^\nu$ given the volumes of the whole family of polytopes $\{ H_{\alpha \beta}^\gamma \}_{\alpha, \beta, \gamma \in \tot}$.  In the terminology of geometric quantization, $\CJ(\lambda, \mu; \nu)$ is a semiclassical approximation of $C_{\lambda \mu}^\nu$, so that an expression for $C_{\lambda \mu}^\nu$ in terms of $\CJ$ can be interpreted as exactly recovering a ``quantum'' object from its classical limit.  General methods for counting lattice points in polytopes based on this type of volume data have been developed by Brion and Vergne in \cite{BV} and by Szenes and Vergne in \cite{SV}, but here we take a different approach.

Our methods are based on the idea of box spline deconvolution.  It follows from a formula of De Concini, Procesi and Vergne \cite{DPV} that for certain fixed values of $\alpha$ and $\beta$, $\CJ(\alpha, \beta ; \gamma)$ can be represented as a convolution of two measures on $\tot$: a finitely supported measure that encodes the multiplicities $C_{\lambda \mu}^\nu$, and a continuous measure called a box spline (see Definition \ref{def:BS}).  Therefore one way to think about computing $C_{\lambda \mu}^\nu$ from $\CJ$ is as a deconvolution problem.  To this end, it is necessary to study the box spline itself in some detail.  The problem of inverting the convolution with the box spline was studied in depth by Dahmen and Micchelli in the 1980's \cite{DM2, DM} and more recently by several geometers and representation theorists \cite{DPV, DV, V-BS}, who introduced the idea of using box spline deconvolution to study representation-theoretic multiplicity problems.  Here we build on a number of these authors' results, which we review below.

The main new technique in this paper, introduced in section \ref{sec:discrete}, is to simplify the problem by restricting $\CJ$ to a lattice.  By moving to this discrete setting we lose no relevant information, but the deconvolution problem becomes more tractable, leading to new formulae for the multiplicities.  We also obtain a reformulation of the $\CJ$-LR relations in terms of a finite-difference operator $\CD$ that we call the box spline Laplacian.  This finally allows us, in Theorem \ref{thm:fin-dif-inversion}, to derive the following concise expression for Littlewood--Richardson coefficients of $\mathfrak{su}(n)$, which holds when the triple $(\lambda, \mu, \nu)$ lies sufficiently far from a certain hyperplane arrangement in $\tot^3$:
\begin{equation} \label{eqn:fd-inv-intro}
C_{\lambda \mu}^\nu = \sum_{k = 0}^{\lfloor d /2 \rfloor} \Big( -\frac{1}{2} \CD \Big)^k \CJ(\lambda', \mu'; \nu'),
\end{equation}
where $d = \frac{1}{2}(n-1)(n-2)$ and the primes indicate the shift by the Weyl vector.  The precise sense of (\ref{eqn:fd-inv-intro}), including notational details and the required assumptions on $(\lambda, \mu, \nu)$, is explained below in sections \ref{sec:finite-difference} and \ref{sec:fd-inversion}.

The box spline convolution and deconvolution identities of \cite{DPV} are general statements in index theory, and they presumably could be used to extend the techniques of this paper to the more general problem of decomposing a representation of $G$ into irreducible representations of an arbitrary closed connected subgroup.  We do not pursue this line of reasoning in full, but we illustrate it at the end of the paper by showing how the same techniques can be used to compute weight multiplicities of an irreducible representation of $G$, leading to an analogue of (\ref{eqn:fd-inv-intro}) for Kostka numbers.

\subsection*{Organization of the paper}

In section \ref{sec:J-BS} we define the volume function and the box spline associated to $\gog$, and we review some of their properties.  We then recall a formula of \cite{DPV} that represents the volume function as a convolution with the box spline.  In section \ref{sec:inversion} we review some generalities on box spline deconvolution due to \cite{DM, DPV, DV}.  We discuss in particular the case $\gog = \mathfrak{su}(n)$, which is special among the cases that we consider because it admits a particularly straightforward expression for the deconvolution operator, leading to an integrodifferential formula for the Littlewood--Richardson coefficients.

Most of the statements in sections \ref{sec:J-BS} and \ref{sec:inversion} are fairly direct consequences of the deconvolution theorems of \cite{DM, DPV, DV} and of formulae for the volume function derived in \cite{CMZ, CZ1}.  Accordingly, these two sections should be regarded primarily as a review of known results, with the goal of working out in detail an important special case of the more general considerations in \cite{DPV, DV}.  We have, however, recorded several explicit formulae and derivations that do not seem to have appeared previously in the literature.

In section \ref{sec:consequences}, we use these representations of $\CJ$ and $C_{\lambda \mu}^\nu$ to give new proofs of three known results: the differentiability class of $\CJ$ (Corollary \ref{cor:J-regularity}), the semiclassical asymptotics expressing $\CJ$ as a scaling limit of tensor product multiplicities (Corollary \ref{cor:J-semiclassical}), and a theorem of Rassart \cite{Rass} and Derksen--Weyman \cite{DW} stating that for fixed $(\lambda, \mu, \nu)$, the Littlewood--Richardson coefficient $C_{N \lambda\, N\mu}^{N\nu}$ of $\mathfrak{su}(n)$ is a polynomial in the variable $N \in \N$ (Theorem \ref{thm:polynomiality}).  We include these proofs mainly for illustrative purposes: the first two demonstrate the wide-ranging consequences of the De Concini--Procesi--Vergne convolution formula, while the third provides intuition for geometric arguments that we will use later in the proof of Theorem \ref{thm:fin-dif-inversion}.

Section \ref{sec:discrete} contains our main new results.  Here we replace the volume function and the box spline with discrete approximations and study convolutions on the root lattice or the weight lattice rather than on the entire Cartan subalgebra.  This simplifies the deconvolution problem; in fact, for any given $(\lambda, \mu, \nu)$, only finitely many values of $\CJ$ are needed to compute the tensor product multiplicity.  We give two methods for calculating $C_{\lambda \mu}^\nu$ in this discrete setting: an algorithm (Theorem \ref{thm:C-from-J-algo}) and an integral formula (Theorem \ref{thm:C-from-J-Fourier}).  Since $\CJ$ can be expressed in terms of Harish-Chandra orbital integrals (defined in (\ref{eqn:H-def}) below), for $\gog = \mathfrak{su}(n)$ these results give the Littlewood--Richardson coefficients in terms of the HCIZ integral (defined in (\ref{eqn:hciz}) below), a widely studied function in random matrix theory.

Next we derive some identities for the discretized box spline on the root lattice, and we introduce a finite difference operator called the box spline Laplacian that provides a convenient representation of discrete convolution with the box spline (Proposition \ref{prop:b-conv-D-rep}).  As a consequence, we obtain an alternate formulation of the $\CJ$-LR relations.  Finally, we draw on results from all of the previous sections to prove a formula for Littlewood--Richardson coefficients as a linear combination of values of $\CJ$ (Theorem \ref{thm:fin-dif-inversion}), which holds for ``typical'' dominant weights $(\lambda, \mu, \nu)$ of $\mathfrak{su}(n)$.

Section 6 sketches how the techniques in the preceding sections are equally applicable to the simpler problem of computing weight multiplicities.

\section{The volume function and the box spline}
\label{sec:J-BS}

Let $G$ be a compact, semisimple, connected, simply connected Lie group of rank $r$ with Lie algebra $\mathfrak{g}$, and $\tot \subset \gog$ a Cartan subalgebra with Weyl group $W$.  Fix a $G$-invariant inner product $\langle \cdot, \cdot \rangle$ identifying $\mathfrak{g} \cong \mathfrak{g}^*$, as well as a choice $\Phi^+ \subset \tot$ of positive roots.

We will study a function $\CJ$ associated to $\gog$, called the volume function.  It is defined in terms of Harish-Chandra orbital integrals,
\begin{equation} \label{eqn:H-def}
\mathcal{H}(x, y) := \int_G e^{\langle \mathrm{Ad}_g y, x \rangle} dg, \qquad x, y \in \tot \otimes \C,
\end{equation}
where $dg$ is the normalized Haar measure.
These integrals admit an exact expression due to Harish-Chandra \cite{HC}:
\begin{equation} \label{eqn:hc-this-context} \Delta_\mathfrak{g}(x) \Delta_\mathfrak{g}(y) \mathcal{H}(x, y) = \Delta_\mathfrak{g}(\rho) \sum_{w \in W} \epsilon(w) e^{\langle w(y), x \rangle}, \qquad x, y \in \mathfrak{t}, \end{equation} where $\Delta_\gog(x) := \prod_{\alpha \in \Phi^+}\langle \alpha, x \rangle$ is the discriminant of $\gog$, $\epsilon(w)$ is the sign of $w \in W$, and $\rho := \frac{1}{2}\sum_{\alpha \in \Phi^+} \alpha$ is the Weyl vector.

\begin{example} \label{ex:hciz} \normalfont
An important special case\footnote{Although $U(N)$ is not semisimple, (\ref{eqn:hciz}) follows readily from (\ref{eqn:hc-this-context}) with $G=SU(N)$.} of (\ref{eqn:hc-this-context}) is the Harish-Chandra--Itzykson--Zuber (HCIZ) integral over the unitary group $U(N)$,
\begin{equation} \label{eqn:hciz} \int_{U(N)} e^{\mathrm{tr} (AUBU^\dagger)} dU = \left( \prod_{p=1}^{N-1}p! \right) \frac{\det(e^{a_i b_j})_{i,j = 1}^N}{\Delta(A) \Delta(B)}, \end{equation}
where $A$ and $B$ are $N$-by-$N$ Hermitian matrices with eigenvalues $a_1 > \hdots > a_N$ and $b_1 > \hdots > b_N$ respectively, and $\Delta(A) = \prod_{i < j} (a_i - a_j)$ is the Vandermonde determinant \cite{IZ}. \remdone
\end{example}

\begin{definition} \label{def:J-def} \normalfont
For $\alpha, \beta, \gamma \in \tot$ we define the {\it volume function} as
\begin{equation} 
\label{eqn:J-def}
\CJ(\alpha, \beta ; \gamma) := \frac{\Delta_\mathfrak{g}(\alpha) \Delta_\mathfrak{g}(\beta) \Delta_\mathfrak{g}(\gamma)}{(2\pi)^r \, |W| \, \Delta_\mathfrak{g}(\rho)^3} \int_\mathfrak{t} \Delta_\mathfrak{g}(x)^2 \mathcal{H}(ix, \alpha) \mathcal{H}(ix, \beta) \mathcal{H}(ix, -\gamma) \, dx.
\end{equation}
It is a homogeneous piecewise polynomial function of the triple $(\alpha, \beta, \gamma) \in \tot^3$, of degree $|\Phi^+| - r$.  We usually take $\alpha, \beta$ fixed and regard $\CJ(\alpha, \beta ; \gamma)$ as a $W$-skew-invariant function of $\gamma \in \tot$, in which case it is supported on a union of $|W|$ convex polytopes in $\tot$. \remdone
\end{definition}

The volume function arises naturally when studying a probabilistic generalization of Horn's problem, which asks about the possible eigenvalues of a sum of two Hermitian matrices whose eigenvalues are known.  Concretely, let $\mathcal{C}_+ \subset \tot$ be the (closed) dominant Weyl chamber, and let $\alpha, \beta \in \mathcal{C}_+$ with $\Delta_\gog(\alpha), \Delta_\gog(\beta) \not = 0$. Suppose $A \in \mathcal{O}_\alpha$, $B \in \mathcal{O}_\beta$ are drawn uniformly at random from the respective coadjoint orbits of $\alpha$ and $\beta$.  Then
\begin{equation} \label{eqn:p-from-J}
p(\gamma | \alpha, \beta) := \frac{\Delta_\gog(\gamma) \Delta_\gog(\rho)}{\Delta_\gog(\alpha) \Delta_\gog(\beta)} \CJ(\alpha, \beta; \gamma), \qquad \gamma \in \mathcal{C}_+
\end{equation}
is the probability density of the random $\gamma \in \mathcal{C}_+$ such that $A + B \in \mathcal{O}_\gamma$, with respect to Lebesgue measure on $\mathcal{C}_+$.

In addition to this probabilistic interpretation, $\CJ(\alpha, \beta ; \gamma)$ can also be interpreted geometrically both as the volume of a symplectic reduction of the product of coadjoint orbits $\mathcal{O}_\alpha \times \mathcal{O}_\beta \times \mathcal{O}_{-\gamma}$ and as the volume of a convex polytope constructed by Berenstein and \mbox{Zelevinsky \cite{BZ}}, the integer points of which count tensor product multiplicities.  We refer the reader to \cite{CMZ, CMZ2, CZ1} for further background on the volume function and for details of its probabilistic and geometric interpretations.

As a first illustration of the relationship between $\CJ$ and the tensor product multiplicities $C_{\lambda \mu}^\nu := \dim \mathrm{Hom}_\gog(V_\lambda \otimes V_\mu \to V_\nu)$, we recall one of the $\CJ$-LR relations shown in \cite[prop. 2]{CMZ}.  These relations are a pair of identities that express certain values of $\CJ$ in terms of $C_{\lambda \mu}^\nu$.

Let $Q \subset \tot$ be the root lattice.  We will say that a triple $(\lambda, \mu, \nu)$ of dominant weights of $\gog$ is {\it compatible} if $\lambda + \mu - \nu \in Q$, as this is a well-known necessary condition for $C_{\lambda \mu}^\nu \ne 0$.  Let a prime denote the shift of a weight by the Weyl vector: $\lambda' = \lambda + \rho.$  Let $K \subset Q$ be the set of dominant elements of the root lattice that lie on the interior of the convex hull of the Weyl orbit of $\rho$, and for $\kappa \in K$ define $r_\kappa := \CJ(\rho, \rho; \kappa')$.  Then for $(\lambda, \mu, \nu)$ compatible, we have \mbox{\cite[eqn. 28]{CMZ}}:
\begin{equation} \label{eqn:JLR1}
\CJ(\lambda', \mu'; \nu') = \sum_{\kappa \in K} \sum_{\substack{\tau \in \lambda + \mu + Q \\ \ \ \cap \ \mathcal{C}_+}} r_\kappa C_{\lambda \mu}^\tau C_{\tau \kappa}^\nu = \sum_{\kappa \in K} r_\kappa C_{\lambda \mu \kappa}^\nu,
\end{equation}
where $C_{\lambda \mu \kappa}^\nu := \dim \mathrm{Hom}_\gog(V_\lambda \otimes V_\mu \otimes V_\kappa \to V_\nu)$ is the multiplicity of $V_\nu$ in the triple tensor product.  The index $\tau$ runs over all dominant elements of the translated root lattice $\lambda + \mu + Q$.  The second $\CJ$-LR relation \cite[eqn. 29]{CMZ} gives a similar formula for $\CJ(\lambda, \mu ; \nu)$ with unshifted arguments.

\begin{remark} \label{rem:conj-property} \normalfont
The $\CJ$-LR relations offer a starting point for deducing properties of $\CJ$ from those of $C_{\lambda \mu}^\nu$ and vice versa.  For example, in \cite{CZ0} Coquereaux and Zuber showed that the total multiplicity of a tensor product of two irreducible representations is unchanged by conjugating one of the highest weights:
\begin{equation} \label{eqn:conj-prop-C}
\sum_{\nu \in \lambda + \bar \mu + Q} C_{\lambda \bar \mu}^\nu = \sum_{\nu \in \lambda + \mu + Q} C_{\lambda \mu}^\nu,
\end{equation}
where $\bar \mu := -w_0(\mu)$ for $w_0$ the longest element of $W$.  Summing over $\nu$ in the $\CJ$-LR relations (\ref{eqn:JLR1}) and \cite[eqn. 29]{CMZ} and using (\ref{eqn:conj-prop-C}), we obtain analogous properties for $\CJ$:

\begin{eqnarray}
\label{eqn:conj-prop-J1} \sum_{\substack{\nu \in \lambda + \bar \mu + Q \\ \ \ \cap \ \mathcal{C}_+}}\CJ(\lambda', \bar \mu' ; \nu') & = & \sum_{\substack{\nu \in \lambda + \mu + Q \\ \ \ \cap \ \mathcal{C}_+}} \CJ(\lambda', \mu' ; \nu'), \\
\label{eqn:conj-prop-J2} \sum_{\substack{\nu \in \lambda + \bar \mu + Q \\ \ \ \cap \ \mathcal{C}_+}}\CJ(\lambda, \bar \mu ; \nu) & = & \sum_{\substack{\nu \in \lambda + \mu + Q \\ \ \ \cap \ \mathcal{C}_+}} \CJ(\lambda, \mu ; \nu).
\end{eqnarray}
\remdone
\end{remark}

The present paper grew out of a desire to answer a question posed by Coquereaux and Zuber in \cite{CZ1}: is there an ``inverse'' formula to (\ref{eqn:JLR1}) that expresses $C_{\lambda \mu}^\nu$ in terms of $\CJ$?  We will find that in fact we can compute $C_{\lambda \mu}^\nu$ from $\CJ$ in multiple ways.  All of them involve a type of piecewise polynomial measure called a box spline.

\begin{definition} \label{def:BS} \normalfont
For $X \subset \tot$ a finite collection of vectors, we define a measure $B_c[X]$ on $\tot$ by
\begin{equation}
\label{eqn:BS-def}
\int_{\mathfrak{t}} f \ dB_c[X] = \int_{-1/2}^{1/2} \cdots \int_{-1/2}^{1/2} f \big( \sum_{\alpha \in X} t_\alpha \alpha \big) \prod_{\alpha \in X} dt_\alpha, \qquad f \in C^0(\tot).
\end{equation}
This measure $B_c[X]$ is the {\it centered box spline} associated to the set $X$. \remdone
\end{definition}

We will mainly study $B_c[\Phi^+]$, which is a probability measure supported on the convex hull of the Weyl orbit of $\rho$.  It has a density $b$ with respect to the Lebesgue measure $dx$ on $\tot$ induced by the inner product $\langle \cdot, \cdot \rangle$, so that
$$\int_{\mathfrak{t}} f \ dB_c[\Phi^+] = \int_\tot f(x)b(x) \, dx, \qquad f \in C^0(\tot).$$
The density $b$ is a piecewise polynomial function of degree $|\Phi^+| - r$.  From the definition (\ref{eqn:BS-def}) we find the following symmetries:
\begin{eqnarray}
\label{eqn:b-sign}
b(-x) &=& b(x), \qquad x \in \tot, \\
\label{eqn:b-W-invariant}
b(w(x)) &=& b(x), \qquad x \in \tot, \ w \in W.
\end{eqnarray}
For a general introduction to box splines, see \cite{dBH, PBP}.

We take the probabilist's sign convention for the Fourier transform of a Borel measure $\mu$ on $\mathfrak{t}$, $$\mathscr{F}[\mu](x) := \int_\mathfrak{t} e^{i \langle \xi, x \rangle} d\mu(\xi), \qquad x \in \tot.$$ A direct calculation then reveals that
\begin{equation} \label{eqn:BS-FT}
\mathscr{F}\big [B_c[\Phi^+] \big ](x) = \prod_{\alpha \in \Phi^+} \frac{e^{i \langle \alpha,x \rangle/2} - e^{-i \langle \alpha, x \rangle/2}}{i \langle \alpha, x \rangle} = j_\mathfrak{g}^{1/2}(x),
\end{equation}
where $j_\gog$ is the function sometimes called the Jacobian of the exponential map, well known to representation theorists due to its appearance in the Kirillov character formula (see (\ref{eqn:KCF}) below).

The box spline $B_c[\Phi^+]$ controls the relationship between the volume function and the tensor product multiplicities.  This is a consequence of a very general formula in index theory due to De Concini, Procesi and Vergne \cite[prop. 5.14]{DPV} (see also \cite[sect. 3]{DV}), which relates the box spline to a quantity called the infinitesimal index of a transversally elliptic symbol.  As a special case, their formula implies the following expression for $\CJ$ as the convolution of $B_c[\Phi^+]$ and a finitely supported $W$-skew-invariant measure that encodes the coefficients $C_{\lambda \mu}^\nu$.

Let $\delta_x$ denote the measure assigning unit mass to the point $x \in \tot$.  Where it improves ease of reading, we will sometimes write $f(x) * \mu$ rather than $(f * \mu)(x)$ for the convolution of a function $f$ and a measure $\mu$, that is, $$f(x) * \mu = \int_\tot f(x - y) \, d\mu(y).$$

\begin{proposition}[De Concini--Procesi--Vergne]
\label{prop:J-BS}
Let $\lambda, \mu \in \tot$ be dominant weights of $\gog$.  Then
\begin{equation} \label{eqn:conv-conclusion}
\mathcal{J}(\lambda', \mu'; \gamma) = b(\gamma) \, * \, \Bigg ( \sum_{\substack{\nu \in (\lambda + \mu) + Q \\ \ \ \cap \ \mathcal{C}_+}} C^\nu_{\lambda \mu} \sum_{w \in W} \epsilon(w) \delta_{w(\nu')} \Bigg ).
\end{equation}
\end{proposition}

Recall that $Q$ denotes the root lattice and $\mathcal{C}_+$ the dominant Weyl chamber, so that the sum over $\nu$ in (\ref{eqn:conv-conclusion}) runs over dominant weights satisfying the compatibility criterion $\lambda + \mu - \nu \in Q$.

Although one could deduce Proposition \ref{prop:J-BS} from \cite[prop. 5.14]{DPV} by invoking index-theoretic constructions, it is more instructive to give a detailed proof specialized to the context that we consider here.  The argument below follows a method sketched in \cite[lem. 5.2]{PV}.

\begin{remark} \label{rem:J-cont} \normalfont
 If $\gog$ contains no simple summands isomorphic to $\mathfrak{su}(2)$, then $\CJ$ and $b$ are both continuous functions of $\gamma \in \tot$.  In these cases, (\ref{eqn:conv-conclusion}) and all equations relating $\CJ$ and $b$ below hold pointwise on $\tot$.  If $\gog$ does contain one or more $\mathfrak{su}(2)$ summands, then both $\CJ$ and $b$ have jump discontinuities on the boundaries of their respective supports, and all expressions for $\CJ$ in terms of $b$ should be understood to hold almost everywhere with respect to Lebesgue measure.  Away from these boundary discontinuities we identify $\CJ$ and $b$ with their locally continuous versions, so that we may talk about their pointwise values at all other points of $\tot$. \remdone
\end{remark}

\begin{proof}[Proof of Proposition \ref{prop:J-BS}]

Recall the Kirillov character formula for compact connected Lie groups \cite{AK}: \begin{equation} \label{eqn:KCF} j_\mathfrak{g}^{1/2}(x) \chi_\lambda(e^x) = \int_{\mathcal{O}_{\lambda'}} e^{i \langle \xi, x \rangle} d\beta_{\mathcal{O}_{\lambda'}} (\xi), \quad x \in \mathfrak{t}. \end{equation}  Here $\chi_\lambda$ is the irreducible character of $G$ with highest weight $\lambda$, $\mathcal{O}_{\lambda'}$ is the coadjoint orbit of $\lambda'$, and $d\beta$ is the Liouville measure of the Kostant--Kirillov--Souriau symplectic form on the coadjoint orbit.

Consider now the direct product $G \times G$.  This group is also compact and connected, and its irreducible representations take the form $V_\lambda \otimes V_\mu$, where $V_\lambda$ and $V_\mu$ are irreducible representations of $G$.  Denote the character of such a representation by $\tilde \chi_{\lambda \otimes \mu}$.  The diagonal subgroup $G_\Delta \subset G \times G$ acts on $V_\lambda \otimes V_\mu$ by the usual representation of $G$ on the tensor product, so identifying $G \cong G_\Delta$ we have $\tilde \chi_{\lambda \otimes \mu} |_{G_\Delta} = \chi_\lambda \chi_\mu$.  To compute the decomposition of $V_\lambda \otimes V_\mu$ into irreducible representations $V_\nu$, it thus suffices to decompose $\tilde \chi_{\lambda \otimes \mu} |_{G_\Delta}$ into irreducible characters $\chi_\nu$.

We start by using the Kirillov character formula (\ref{eqn:KCF}) to obtain an expression for $\tilde \chi_{\lambda \otimes \mu}$.  The coadjoint orbit of $G \times G$ corresponding to the irreducible representation $V_\lambda \otimes V_\mu$ is $\mathcal{O}_{\lambda'} \times \mathcal{O}_{\mu'}$.  The positive roots of $G \times G$ are $\Phi^+ \sqcup \Phi^+$, i.e.~we count each positive root of $G$ twice, once for each factor, so that $j^{1/2}_{\mathfrak{g} \oplus \mathfrak{g}}(x,y) = j^{1/2}_\mathfrak{g}(x)j^{1/2}_\mathfrak{g}(y)$.  Putting all this into (\ref{eqn:KCF}), we get: \begin{multline*} j^{1/2}_\mathfrak{g}(x) j^{1/2}_\mathfrak{g}(y) \tilde \chi_{\lambda \otimes \mu} (e^{(x,y)}) = \int_{\mathcal{O}_{\lambda'} \times \mathcal{O}_{\mu'}} e^{i (\langle \xi, x \rangle + \langle \eta, y \rangle)} d\beta_{\mathcal{O}_{\lambda'} \times \mathcal{O}_{\mu'}} (\xi,\eta), \\ (x,y) \in \mathfrak{t} \oplus \mathfrak{t}. \end{multline*}

Restricting to the diagonal $x=y$, this becomes: \begin{multline} \label{eqn:restrict-diag} j_\mathfrak{g}(x) \tilde \chi_{\lambda \otimes \mu} (e^{(x,x)}) = j_\mathfrak{g}(x) \chi_{\lambda} (e^x) \chi_\mu(e^x) = j_\mathfrak{g}(x) \sum_\nu C^\nu_{\lambda \mu} \chi_\nu(e^x) \\ = \int_{\mathcal{O}_{\lambda'} \times \mathcal{O}_{\mu'}} e^{i \langle \xi + \eta, x \rangle} d\beta_{\mathcal{O}_{\lambda'} \times \mathcal{O}_{\mu'}} (\xi,\eta), \quad x \in \mathfrak{t}. \end{multline}

Using the fact that the symplectic volume of $\mathcal{O}_{\nu'}$ is equal to $\Delta_\mathfrak{g}(\nu')/{ \Delta_\mathfrak{g}(\rho)}$ (see e.g.~\cite[sect.~4]{McS}), we can reduce the integral over $\mathcal{O}_{\lambda'} \times \mathcal{O}_{\mu'}$ to two integrals over $G$ with respect to the normalized Haar measure $dg$, so that (\ref{eqn:restrict-diag}) becomes: \begin{equation} \label{eqn:RHS-G-int} j_\mathfrak{g}(x) \sum_\nu C^\nu_{\lambda \mu} \chi_\nu(e^x)  = \frac{\Delta_\mathfrak{g}(\lambda') \Delta_\mathfrak{g}(\mu')}{\Delta_\mathfrak{g}(\rho)^2} \mathcal{H}(ix, \lambda') \mathcal{H}(ix, \mu').\end{equation}

Applying Kirillov's formula (\ref{eqn:KCF}) again to each $\chi_\nu$ and then using the Harish-Chandra integral formula (\ref{eqn:hc-this-context}), we find that the left-hand side of (\ref{eqn:RHS-G-int}) can be rewritten as:
\begin{eqnarray} \label{eqn:LHS-G-int}
\nonumber
j_\mathfrak{g}(x) \sum_\nu C^\nu_{\lambda \mu} \chi_\nu(e^x) &=& j_\mathfrak{g}^{1/2}(x) \sum_\nu C^\nu_{\lambda \mu} \int_{\mathcal{O}_{\nu'}} e^{i\langle \xi, x \rangle} d\beta_{\mathcal{O}_{\nu'}}(\xi) \\
\nonumber
&=& j_\mathfrak{g}^{1/2}(x) \sum_\nu C^\nu_{\lambda \mu} \frac{\Delta_\mathfrak{g}(\nu')}{\Delta_\mathfrak{g}(\rho)} \mathcal{H}(ix, \nu') \\
&=& \frac{j_\mathfrak{g}^{1/2}(x)}{\Delta_\mathfrak{g}(ix)} \sum_\nu C^\nu_{\lambda \mu} \sum_{w \in W}\epsilon(w)e^{i\langle w(\nu'), x \rangle}.\end{eqnarray}
Equating this last expression to the right-hand side of (\ref{eqn:RHS-G-int}) and multiplying through by $\Delta_\mathfrak{g}(ix)$, we finally obtain: \begin{multline} \label{eqn:last-before-IFT} j_\mathfrak{g}^{1/2}(x) \sum_\nu C^\nu_{\lambda \mu} \sum_{w \in W}\epsilon(w)e^{i\langle w(\nu'), x \rangle} \\ = \frac{\Delta_\mathfrak{g}(\lambda') \Delta_\mathfrak{g}(\mu')}{\Delta_\mathfrak{g}(\rho)^2} \Delta_\mathfrak{g}(ix) \mathcal{H}(ix, \lambda') \mathcal{H}(ix, \mu'). \end{multline}

Now we take the inverse Fourier transform, over $\mathfrak{t}$, of each side of (\ref{eqn:last-before-IFT}).  On the left-hand side, using (\ref{eqn:BS-FT}) we find:
\begin{multline} \label{eqn:LHS-IFT-t}
\mathscr{F}^{-1} \left[ j_\mathfrak{g}^{1/2}(x) \sum_\nu C^\nu_{\lambda \mu} \sum_{w \in W}\epsilon(w)e^{i\langle w(\nu'), x \rangle} \right](\gamma) \\ = b(\gamma) \, * \, \left( \sum_\nu C_{\lambda \mu}^\nu \sum_{w \in W} \epsilon(w) \delta_{w(\nu')} \right).
\end{multline}
On the right-hand side, we have:
\begin{align} 
\nonumber \mathscr{F}^{-1} &\left[ \frac{\Delta_\mathfrak{g}(\lambda') \Delta_\mathfrak{g}(\mu')}{\Delta_\mathfrak{g}(\rho)^2} \Delta_\mathfrak{g}(ix) \mathcal{H}(ix, \lambda') \mathcal{H}(ix, \mu') \right](\gamma) \\
\nonumber &= \frac{1}{(2\pi)^r} \frac{\Delta_\mathfrak{g}(\lambda') \Delta_\mathfrak{g}(\mu')}{\Delta_\mathfrak{g}(\rho)^2} \int_\mathfrak{t} \Delta_\mathfrak{g}(ix) \mathcal{H}(ix, \lambda') \mathcal{H}(ix, \mu') e^{-i \langle x, \gamma \rangle} dx \\
\nonumber &= \frac{\Delta_\mathfrak{g}(\lambda') \Delta_\mathfrak{g}(\mu')}{(2\pi)^r\, |W|\, \Delta_\mathfrak{g}(\rho)^2} \int_\mathfrak{t} \Delta_\mathfrak{g}(ix) \mathcal{H}(ix, \lambda') \mathcal{H}(ix, \mu') \left( \sum_{w \in W} \epsilon(w) e^{\langle w(ix), -\gamma \rangle} \right) dx \\
\label{eqn:RHS-IFT-t} &= \frac{\Delta_\mathfrak{g}(\lambda') \Delta_\mathfrak{g}(\mu') \Delta_\mathfrak{g}(\gamma)}{(2\pi)^r\, |W|\, \Delta_\mathfrak{g}(\rho)^3} \int_\mathfrak{t} \Delta_\mathfrak{g}(x)^2 \mathcal{H}(ix, \lambda') \mathcal{H}(ix, \mu') \mathcal{H}(ix, -\gamma) \, dx,
\end{align}
where in the last line we have again applied the Harish-Chandra formula (\ref{eqn:hc-this-context}) and have used the fact that $\Delta_\mathfrak{g}$ is homogeneous of degree $|\Phi^+| = (\dim \mathfrak{g} - r)/2$ to cancel factors of $-1$ and $i$.

Comparing (\ref{eqn:RHS-IFT-t}) to the definition (\ref{eqn:J-def}) of $\mathcal{J}$, we see that this last expression is equal to $\CJ(\lambda', \mu' ; \gamma)$, completing the proof.
\end{proof}

To conclude this section, we quickly derive a more compact expression for $\CJ$.  For $\gog = \mathfrak{su}(n)$, the identity below reduces to \cite[eqn. 18]{Z}.

\begin{proposition}
\label{prop:J-compact}
Define the function $\psi_{\alpha \beta}: \tot \to \C$ by
\begin{equation} \label{eqn:psi-def}
\psi_{\alpha \beta}(x) = (-i)^{|\Phi^+|} \sum_{w, w' \in W} \frac{\epsilon(ww')}{\Delta_\mathfrak{g}(x)} e^{i\langle x, w(\alpha) + w'(\beta) \rangle}.
\end{equation}
Then $\CJ(\alpha, \beta ; \gamma) = \mathscr{F}^{-1}[\psi_{\alpha \beta}](\gamma)$.
\end{proposition}

\begin{remark} \label{rem:integrand-interp} \normalfont
Some care is required in interpreting both the definition of $\psi_{\alpha \beta}$ and the inverse Fourier transform in Proposition \ref{prop:J-compact}.  For $x$ such that $\Delta_\mathfrak{g}(x) = 0$, we understand the expression (\ref{eqn:psi-def}) as the limit $\psi_{\alpha \beta}(x) := \lim_{t \to 0}\psi_{\alpha \beta}(x + t \rho)$; the double sum over the Weyl group then introduces cancelations such that $\psi_{\alpha \beta}$ vanishes.  Moreover when $\gog$ contains a simple summand isomorphic to $\mathfrak{su}(2)$, the usual integral representation of the inverse Fourier transform of $\psi_{\alpha \beta}$ is not absolutely convergent and must instead be interpreted as a Cauchy principal value. \remdone
\end{remark}

\begin{proof}
We may assume that $\Delta_\gog(\alpha)$ and $\Delta_\gog(\beta)$ are both nonzero, since otherwise $\CJ$ and $\psi_{\alpha \beta}$ both vanish and there is nothing to prove.  Applying Harish-Chandra's formula (\ref{eqn:hc-this-context}) to the definition (\ref{eqn:J-def}) and canceling factors of $\Delta_\gog$, we can rewrite
\begin{multline*}
\mathcal{J}(\alpha, \beta; \gamma) = \\
\frac{(-i)^{|\Phi^+|}}{(2\pi)^r\, |W|} \int_\mathfrak{t} \frac{1}{\Delta_\mathfrak{g}(x)} \ \left( \sum_{w \in W} \epsilon(w) e^{i \langle w(x), \alpha \rangle} \right) \left( \sum_{w' \in W} \epsilon(w') e^{i \langle w'(x), \beta \rangle} \right) \left( \sum_{w'' \in W} \epsilon(w'') e^{-i \langle w''(x), \gamma \rangle} \right) dx \\
= \frac{(-i)^{|\Phi^+|}}{(2\pi)^r\, |W|} \int_\mathfrak{t} \sum_{w, w', w'' \in W} \frac{\epsilon(ww'w'')}{\Delta_\mathfrak{g}(x)} e^{i\langle x, w(\alpha) + w'(\beta) - w''(\gamma) \rangle} dx,
\end{multline*}
where the integral is interpreted according to the discussion in Remark \ref{rem:integrand-interp}.

Using the $W$-invariance of the inner product and the $W$-skewness of $\Delta_\mathfrak{g}$, we can reindex the triple sum as a double sum, giving
\begin{eqnarray} \nonumber \mathcal{J}(\alpha, \beta; \gamma) &=& \frac{(-i)^{|\Phi^+|}}{(2\pi)^r} \int_\mathfrak{t}  \sum_{w, w' \in W} \frac{\epsilon(ww')}{\Delta_\mathfrak{g}(x)} e^{i\langle x, w(\alpha) + w'(\beta) - \gamma \rangle} dx \\
\label{eqn:J-rewritten} &=& \mathscr{F}^{-1}[\psi_{\alpha \beta}](\gamma) \end{eqnarray}
as desired.
\end{proof}

\section{Unimodularity and the $\hat A(\Phi^+)$ operator}
\label{sec:inversion}

As we will see below in section \ref{sec:discrete}, the coefficients $C_{\lambda \mu}^\nu$ can always be recovered from $\CJ$.  In other words, the convolution in Proposition \ref{prop:J-BS} is invertible.  However, when $\gog = \mathfrak{su}(n)$, the inverse admits a particularly convenient form that is more explicit than in most other cases.  This is due to the fact that the positive roots of $\mathfrak{su}(n)$ are {\it unimodular},\footnote{This use of the term ``unimodular'' is not to be confused with the different notion of a unimodular (i.e.~self-dual) lattice.  Indeed the root lattice of $\mathfrak{su}(n)$ is not unimodular.} meaning that any collection of positive roots spanning $\tot$ also generates the root lattice $Q$.  Direct calculations with the other classical and exceptional root systems reveal that the series $\mathfrak{su}(n)$ are the only compact simple Lie algebras with this property.

Unimodularity of $\Phi^+$ implies -- in a delicate sense that we will shortly make precise -- that convolution with $B_c[\Phi^+]$ can be inverted by a differential operator.  By applying this operator to $\CJ$ and taking a limit, we can compute the coefficients $C_{\lambda \mu}^\nu$.  This result is an application of a deconvolution formula originally proved by Dahmen and Micchelli in \cite{DM} and dramatically expanded by Vergne and collaborators in \cite{BV, DPV, DV, SV, V-BS}.  In fact similar deconvolution results hold even when the positive roots are not unimodular, however in these situations the formulae in terms of differential operators do not allow $C_{\lambda \mu}^\nu$ to be recovered from $\CJ$ alone but rather require knowledge of a larger family of piecewise polynomial functions corresponding to Duistermaat--Heckman measures for various submanifolds of $\mathcal{O}_{\lambda'} \times \mathcal{O}_{\mu'}$.  Accordingly we do not treat such cases here and instead refer the reader to \cite{DV, V-BS} for the relevant results.  In section \ref{sec:discrete} we will develop a different approach to deconvolution for both the unimodular and non-unimodular cases.

To start, we review some generalities on convolution with the box spline.  Let $\C^Q$ be the space of complex-valued functions on $Q$ and let $L^1_\mathrm{loc}(\tot)$ be the space of locally integrable functions on $\tot$. Consider the operator \mbox{$\mathcal{T} : \C^Q \to L^1_\mathrm{loc}(\tot)$} defined by
\begin{equation} \label{eqn:T-def}
\mathcal{T}m(\gamma) := b(\gamma) \, * \sum_{\nu \in Q} m(\nu)\, \delta_\nu, \qquad m : Q \to \C.
\end{equation}
If $\mathcal{T}m \ne 0$, then $\mathcal{T} m$ is a piecewise polynomial function of degree $d := |\Phi^+| - r$.

Proposition \ref{prop:J-BS} expresses $\CJ(\lambda', \mu' ; \gamma)$ in the form (\ref{eqn:T-def}), up to an inconsequential translation of $Q$ by $\lambda + \mu$.  Therefore, if we hope to recover $C_{\lambda \mu}^\nu$ from $\CJ$, the first question we should ask is whether $\mathcal{T}$ is injective.  The answer depends on whether or not $\Phi^+$ is unimodular. In particular, from \cite[thm. 3.2]{DM2} and \cite[thm. 4.1]{DM} we have:

\begin{theorem}[Dahmen--Micchelli] \label{thm:T-injective}
The operator $\mathcal{T}$ is injective on $\C^Q$ if and only if $\Phi^+$ is unimodular.  Moreover, if $\Phi^+$ is not unimodular, then the kernel of $\mathcal{T}$ intersects nontrivially with $\ell^\infty(Q)$.
\end{theorem}

Fortunately, to compute $C_{\lambda \mu}^\nu$ from $\CJ$, we only need $\mathcal{T}$ to be injective on finitely supported functions, and indeed this always holds (see section \ref{sec:lattice-deconv}).  However, Theorem \ref{thm:T-injective} makes the unimodular case much simpler to handle.  It allows many nice results for $\gog = \mathfrak{su}(n)$ that do not hold in other cases, including the deconvolution formula that we now develop.

For any $\gog$, we can write down an infinite-order differential operator $\hat A(\Phi^+)$, which inverts $\mathcal{T}$ on a specific class of functions that we define below.  The operator $\hat A(\Phi^+)$ is defined by
\begin{equation} \label{eqn:A-def}
\hat A (\Phi^+) := \prod_{\alpha \in \Phi^+} \frac{\partial_\alpha}{e^{\frac{1}{2} \partial_\alpha} - e^{- \frac{1}{2} \partial_\alpha}},
\end{equation}
where $\partial_\alpha f(x) := \frac{d}{dt} \big |_{t = 0} f(x + t \alpha)$ for $f \in C^1(\tot)$.
We interpret this operator as a series expansion, recognizing that we can write
$$\hat A(\Phi^+) = \prod_{\alpha \in \Phi^+} \hat a_\alpha(\partial),$$
where
$$\hat a_\alpha(x) := \frac{\langle \alpha, x \rangle}{2} \mathrm{csch}\left( \frac{\langle \alpha, x \rangle}{2} \right)$$
and $\mathrm{csch}$ is the hyperbolic cosecant.  We have the Taylor series
$$\mathrm{csch}(z) = \frac{1}{z} - \sum_{n = 1}^\infty \frac{2(2^{2n-1} - 1) B_{2n}}{(2n)!} z^{2n-1}$$
where
$$B_{2n} = \sum_{k = 0}^{2n} \sum_{j = 0}^k (-1)^j {k \choose j} \frac{j^{2n}}{k+1}, \qquad n = 1, 2, \hdots$$
is the $(2n)^{\mathrm{th}}$ Bernoulli number, with $B_0 = 1$.  This gives
\begin{equation} \label{eqn:Ahat-rewritten}
\hat A(\Phi^+) = \prod_{\alpha \in \Phi^+} \left( \sum_{n=0}^\infty \frac{(2^{1-2n} - 1)B_{2n}}{(2n)!} \partial_\alpha^{2n} \right) = \sum_{\vec n} \left( \prod_{\alpha \in \Phi^+} \frac{(2^{1-2n_\alpha} - 1)B_{2n_\alpha}}{(2n_\alpha)!} \right) \partial^{2\vec n}
\end{equation}
where in this last expression $\vec n = (n_\alpha)$ runs over multi-indices with $|\Phi^+|$ components, and $\partial^{2\vec n} := \prod_{\alpha \in \Phi^+} \partial_\alpha^{2 n_\alpha}$.

The significance of $\hat A(\Phi^+)$ is that it allows us to invert $\mathcal{T}$ on suitable polynomials.  Concretely, let $D(\Phi^+)$ denote the space of polynomials in the image of $\mathcal{T}$:
\begin{equation} \label{eqn:DMspace}
D(\Phi^+) := \{ \ p \in \Pi(\tot) \ | \ p = \mathcal{T}m \textrm{ for some } m \in \C^Q \ \},
\end{equation}
where $\Pi(\tot)$ denotes the space of polynomial functions on $\tot$. In \cite{DM2}, Dahmen and Micchelli characterized $D(\Phi^+)$ as the solution space of a certain system of partial differential equations and computed its dimension.\footnote{In fact the results of \cite{DM2, DM} are much more general than what we state here, as they concern arbitrary box splines not necessarily associated to root systems.}  In \cite{DM}, they proved the following (see also \cite[thm. 2.23]{DPV}):
\begin{theorem}[Dahmen--Micchelli] \label{thm:A-T-inverse}
The map $p \mapsto \mathcal{T}(p \big |_Q)$ is a linear isomorphism of $D(\Phi^+)$, with inverse $\hat A(\Phi^+)$.
\end{theorem}

The above holds for any $\gog$, but in the unimodular case where $\mathcal{T}$ is injective not merely on $D(\Phi^+)$ but on all of $\C^Q$, we can use $\hat A(\Phi^+)$ to construct a global inverse of $\mathcal{T}$, giving a general deconvolution formula (see \cite{DM} and also \mbox{\cite[thm. 2.1]{DV}}):

\begin{theorem}[Dahmen--Micchelli, Duflo--Vergne]
\label{thm:bs-deconv}
Suppose that $\Phi^+$ is unimodular and let $m: Q \to \C$.  Choose a vector $\eta$ in the positive cone generated by $\Phi^+$ such that $\eta$ does not lie in any hyperplane spanned by elements of $\Phi^+$.  Then
\begin{equation} \label{eqn:deconv} m(\nu) = \lim_{t \to 0^+} \hat A(\Phi^+) \mathcal{T}m(\nu + t\eta), \qquad \nu \in Q. \end{equation}
\end{theorem}

\begin{remark}
\label{rem:A-truncate} \normalfont
The limit in (\ref{eqn:deconv}) ensures that we only ever evaluate $\hat A(\Phi^+) \mathcal{T}m$ on the interior of a polynomial domain.  This allows us to truncate the series expansion (\ref{eqn:Ahat-rewritten}) so that $\hat A(\Phi^+)$ acts locally as a finite-order operator. \remdone
\end{remark}

Observing that the translation $Q \to \lambda + \mu + Q$ in (\ref{eqn:conv-conclusion}) is inconsequential for the deconvolution formula, we obtain the following corollary to Proposition \ref{prop:J-BS} and Theorem \ref{thm:bs-deconv}, which is the sought-after expression for the Littlewood--Richardson coefficients.
\begin{corollary}
\label{cor:inversion}
Let $\gog = \mathfrak{su}(n)$ and choose $\eta$ as in Theorem \ref{thm:bs-deconv}. For a compatible triple $(\lambda, \mu, \nu)$ of dominant weights, the Littlewood--Richardson coefficient is expressed in terms of the volume function as
\begin{equation} \label{eqn:inversion-formula} C^\nu_{\lambda \mu} = \lim_{t \to 0^+} \hat A(\Phi^+) \mathcal{J}(\lambda', \mu' ; \nu' + t \eta),\end{equation} where the operator $\hat A(\Phi^+)$ acts in the third argument of $\mathcal{J}$. 
\end{corollary}

Note that $\mathcal{J}$ is a piecewise polynomial of degree $d$, so that following Remark \ref{rem:A-truncate} above, we can replace $\hat A(\Phi^+)$ in (\ref{eqn:inversion-formula}) with its series expansion to this same order.  Thus we have expressed $C_{\lambda \mu}^\nu$ as a finite sum of derivatives of $\CJ$.  Putting together (\ref{eqn:Ahat-rewritten}), (\ref{eqn:inversion-formula}) and Proposition \ref{prop:J-compact}, we can write Corollary \ref{cor:inversion} in a more explicit form:

\begin{proposition}
\label{prop:inversion-reworked}
For $\gog = \mathfrak{su}(n)$, $\eta$ as in Theorem \ref{thm:bs-deconv}, $(\lambda, \mu, \nu)$ a compatible triple of dominant weights, and $\psi_{\lambda' \mu'}$ as in (\ref{eqn:psi-def}),
\begin{equation}
\label{eqn:inversion-reworked}
C_{\lambda \mu}^\nu = \lim_{t \to 0^+} \sum_{|\vec n| \le \lfloor d/2 \rfloor} \!\!\!\left( \prod_{\alpha \in \Phi^+} \frac{(2^{1-2n_\alpha} - 1)B_{2n_\alpha}}{(2n_\alpha)!} \right) \partial^{2\vec n} \mathscr{F}^{-1}[\psi_{\lambda' \mu'}](\nu' + t\eta).
\end{equation}
\end{proposition}

\begin{example} \label{ex:Ahat-low-rank} \normalfont
For $\mathfrak{su}(2)$ and $\mathfrak{su}(3)$, $d = 0$ and 1 respectively, so only the degree 0 term of $\hat A(\Phi^+)$ acts non-trivially, and we recover the well-known result (see \cite{CZ1}): 
\begin{equation} \label{Ahat-su2su3}
C_{\lambda \mu}^\nu = \CJ(\lambda', \mu' ; \nu').
\end{equation}
The first non-trivial case is $\mathfrak{su}(4)$, where we have $d=3$, so that we must expand $\hat A(\Phi^+)$ to second order:
\begin{equation} \label{Ahat-su4}
C_{\lambda \mu}^\nu = \lim_{t \to 0^+} \left( 1 - \frac{1}{24} \sum_{\alpha \in \Phi^+} \partial_\alpha^2 \right) \CJ(\lambda', \mu' ; \nu' + t \eta).
\end{equation} \remdone
\end{example}

\section{First applications}
\label{sec:consequences}

In this section we use Propositions \ref{prop:J-BS} and \ref{prop:inversion-reworked} to give new proofs of three known results: the number of continuous derivatives of $\CJ$, the semiclassical limit by which $\CJ(\lambda, \mu; \nu)$ approximates $C_{\lambda, \mu}^\nu$ for $\lambda, \mu, \nu$ large, and a polynomiality property of the Littlewood--Richardson coefficients.  These proofs illustrate the wide range of phenomena that can be understood as consequences of the box spline convolution identity (\ref{eqn:conv-conclusion}).  The considerations of sections \ref{sec:J-reg} and \ref{sec:semiclassical} depend only on Proposition \ref{prop:J-BS} and apply to arbitrary compact semisimple $\gog$, while in section \ref{sec:polynomiality} we take $\gog = \mathfrak{su}(n)$.

\subsection{Regularity of $\CJ$}
\label{sec:J-reg}

In \cite{CMZ, CZ1}, the number of continuous derivatives of $\mathcal{J}$ was determined via Fourier-analytic arguments, essentially by studying the decay of the function $\psi_{\alpha \beta}$ defined in (\ref{eqn:psi-def}). The following corollary to Proposition \ref{prop:J-BS} provides an alternative proof of the degree of differentiability of $\mathcal{J}$.  Recall that the density $b(\gamma)$ of $B_c[\Phi^+]$ is a piecewise polynomial function of degree $d = |\Phi^+| - r$, supported on the convex hull of the Weyl orbit of $\rho$.

\begin{corollary}
\label{cor:J-regularity}
Fix arbitrary $\alpha, \beta \in \tot$.  As a function of $\gamma$, $\CJ(\alpha, \beta; \gamma)$ has at least as many continuous derivatives as $b(\gamma)$.
\end{corollary}

\begin{proof}
By the $W$-skewness of $\CJ$ in all three arguments, it suffices to consider the case that $\alpha, \beta, \gamma$ all lie in the dominant Weyl chamber.  Proposition \ref{prop:J-BS} implies the statement in the case that $\alpha = \lambda'$, $\beta = \mu'$ for $\lambda, \mu$ a pair of dominant weights.

The statement for arbitrary $\alpha, \beta$ then follows by an approximation argument.  First note that if $\gog$ has a simple summand isomorphic to $\mathfrak{su}(2)$, then $b$ is discontinuous and there is nothing to prove.  Thus we may assume $\gog$ contains no $\mathfrak{su}(2)$ summands, so that $b$ is continuous.  For any $\varepsilon > 0$ we can then choose $s_\varepsilon > 0$ and dominant weights $\lambda_\varepsilon, \mu_\varepsilon$ such that
$$ | \CJ(\alpha, \beta, \gamma) - \CJ(\lambda_\varepsilon' /s_\varepsilon , \mu_\varepsilon' / s_\varepsilon ,\gamma) | < \varepsilon.$$
By homogeneity, $$\CJ(\lambda_\varepsilon' /s_\varepsilon , \mu_\varepsilon' / s_\varepsilon ,\gamma) = s_{\varepsilon}^{-d} \CJ(\lambda_\varepsilon', \mu_\varepsilon', s_\epsilon \gamma),$$ so that $\CJ(\lambda_\varepsilon' /s_\varepsilon , \mu_\varepsilon' / s_\varepsilon ,\gamma)$ also has at least the same regularity in $\gamma$ as $b$ does.  Let $D$ be any constant-coefficient differential operator such that $Db$ is continuous.  If we apply $D$ to $\CJ$ in the third argument, then $D \CJ(\lambda' /s , \mu' / s ,\gamma)$ is a continuous function of $\gamma$ whenever $\lambda, \mu$ are dominant weights and $s > 0$.  For this to be the case, clearly $\CJ(\lambda' /s , \mu' / s ,\gamma)$ must also be continuous in $\gamma$, and from the symmetry of the definition (\ref{eqn:J-def}) of $\CJ$, both $\CJ$ and $D\CJ$ must be continuous in their first two arguments as well.  Since points of the form $\lambda' /s $ are dense in the dominant chamber, we find that $D \CJ$ defines a continuous piecewise polynomial function on $\tot^3$.  In particular, $D \CJ(\lambda'_\varepsilon /s_\varepsilon , \mu'_\varepsilon / s_\varepsilon ,\gamma)$ converges uniformly in $\gamma$ as $\varepsilon \to 0$, so that its limit is a continuous function of $\gamma$ and equals the derivative $D\CJ(\alpha, \beta, \gamma)$.
\end{proof}

In general the degree of differentiability of $b$ depends on $\gog$, but it may be determined based on the following standard fact about box splines, which is proved for example in \cite[sect. 1.5]{PBP}:

\begin{theorem} \label{thm:b-regularity}
If all subsets of $\Phi^+$ obtained by deleting $k+1$ roots span $\tot$, then $b \in C^k(\tot)$.
\end{theorem}

Using Theorem \ref{thm:b-regularity}, we can determine the degree of differentiability of $\CJ$ for all of the classical series of simple Lie algebras.

\begin{example} \label{ex:J-regularity} \normalfont
When $\gog = \mathfrak{su}(n)$, the subset $\{ e_i - e_j \ | \ 1 \le i < j \le n-1 \} \subset \Phi^+$, obtained by removing the $n-1$ roots of the form $e_i - e_n$, does not span $\tot$.  On the other hand, a simple induction argument reveals that all subsets of $\Phi^+$ obtained by removing $n - 2$ roots do span $\tot$.  For $\mathfrak{su}(n)$ we thus find that $\CJ \in C^{n-3}(\tot)$ in agreement with the results obtained by Fourier analysis in \cite{CMZ, CZ1}, while we expect in general that $\CJ \not \in C^{n-2}(\tot)$.

When $\gog = \mathfrak{so}(2n)$, the subset $\{ e_i \pm e_j \ | \ 1 \le i < j \le n-1 \} \subset \Phi^+$, obtained by removing the $2n-2$ roots of the form $e_i \pm e_n$, does not span $\tot$, whereas all subsets of $\Phi^+$ obtained by removing $2n - 3$ roots do span $\tot$.  We thus find that $\CJ \in C^{2n-4}(\tot)$, while we expect in general that $\CJ \not \in C^{2n-3}(\tot)$.

For $\gog = \mathfrak{so}(2n+1)$, the subset $\{e_i, \ e_i \pm e_j \ | \ 1 \le i < j \le n-1 \} \subset \Phi^+$, obtained by removing $2n - 1$ roots (the short root $e_n$ and all roots of the form $e_i \pm e_n$), does not span $\tot$, whereas all subsets of $\Phi^+$ obtained by removing $2n - 2$ roots do span $\tot$.  We thus find that $\CJ \in C^{2n-3}(\tot)$, while in general $\CJ \not \in C^{2n-2}(\tot)$.  Again this result agrees with remarks in \cite[sect. 4.4]{CMZ} based on Fourier analysis.

Repeating the argument above but replacing the short roots $e_i$ of $B_n$ with the long roots $2 e_i$ of $C_n$, we find the same result for $\mathfrak{sp}(2n)$ as for $\mathfrak{so}(2n+1)$. \remdone
\end{example}

\subsection{Stretched multiplicities and the semiclassical limit}
\label{sec:semiclassical}
There are various ways in which $\CJ$ can be understood to provide a ``semiclassical approximation'' of the multiplicities $C_{\lambda \mu}^\nu$.  In the framework of geometric quantization \cite{GS}, one can realize $\CJ(\lambda, \mu; \nu)$ as the symplectic volume of the classical phase space of a certain Hamiltonian system, and $C_{\lambda \mu}^\nu$ as the dimension of the state space of a corresponding quantum-mechanical system.  The semiclassical approximation can also be understood in more combinatorial terms: Berenstein and Zelevinsky \cite{BZ} constructed polytopes $H_{\lambda \mu}^\nu$ such that $C_{\lambda \mu}^\nu$ is equal to the number of integer points in $H_{\lambda \mu}^\nu$, while in \cite{CMZ} it was shown that $\CJ(\lambda, \mu; \nu)$ equals the volume of $H_{\lambda \mu}^\nu$, leading to an asymptotic equality between $\CJ(\lambda, \mu; \nu)$ and $C_{\lambda \mu}^\nu$ as $\lambda, \mu, \nu$ grow large.  Here we show how the box spline convolution identity for $\CJ$ provides yet one more perspective on the semiclassical approximation.

Given a triple of highest weights $(\lambda, \mu, \nu)$ with $\lambda + \mu - \nu \in Q$, we can study the {\it stretched multiplicities} $C_{N \lambda \, N \mu}^{N \nu}$ for positive integers $N$.  Much is known about the stretched multiplicities: for example, since the inequalities defining the polytope $H_{\lambda \mu}^\nu$ are linear in $(\lambda, \mu, \nu)$, $C_{N \lambda \, N \mu}^{N \nu}$ is equal to the number of integer points in the dilated polytope $NH_{\lambda \mu}^\nu$. By standard results in the geometry of polytopes, it then follows that the stretched multiplicities for fixed $(\lambda, \mu, \nu)$ are given by a quasi-polynomial function of $N$, and that
\begin{equation} \label{eqn:C-J-sc}
\CJ(\lambda, \mu ; \nu) = \lim_{N \to \infty} \frac{1}{N^d} C_{N \lambda \, N \mu}^{N \nu}.
\end{equation}
This is what is usually meant by the statement that $\CJ$ is a ``semiclassical limit'' of tensor product multiplicities.  The main observation in this subsection is that the asymptotic relationship (\ref{eqn:C-J-sc}) can also be understood in terms of the following exact relationship between $\CJ$ and $C_{N \lambda \, N \mu}^{N \nu}$.

\begin{proposition}
\label{prop:J-stretched}
For $\lambda, \mu$ dominant weights of $\gog$,
\begin{multline}
\label{eqn:J-stretched}
\CJ(\lambda + \rho/N, \mu + \rho/N ; \gamma)  
\\ = \frac{1}{N^d} b(N \gamma) * \Bigg ( \sum_{\substack{\nu \in (\lambda + \mu) + N^{-1}Q \\ \ \ \cap \ \mathcal{C}_+}} C^{N\nu}_{N\lambda\, N\mu} \sum_{w \in W} \epsilon(w)\, \delta_{w(\nu + \rho/N)} \Bigg ),
\end{multline}
where the sum over $\nu$ runs over those elements of the scaled and translated root lattice $(\lambda + \mu) + N^{-1}Q$ that lie in the dominant Weyl chamber $\mathcal{C}_+$. 
\end{proposition}

\begin{proof}
Proposition \ref{prop:J-BS} gives
\begin{align*}
\CJ(N\lambda + \rho, N \mu + \rho ; \gamma) \ &= \ b(\gamma) \, * \, \Bigg( \sum_{\substack{\nu \in (N\lambda + N\mu) + Q \\ \ \ \cap \ \mathcal{C}_+}} C^\nu_{N\lambda \, N\mu} \sum_{w \in W} \epsilon(w) \delta_{w(\nu + \rho)} \Bigg) \\
\ &= \ b(\gamma) \, * \, \Bigg ( \sum_{\substack{\nu \in (\lambda + \mu) + N^{-1}Q \\ \ \ \cap \ \mathcal{C}_+}}  C^{N\nu}_{N\lambda \, N\mu} \sum_{w \in W} \epsilon(w) \delta_{w(N\nu + \rho)} \Bigg ).
\end{align*}
Scaling $\gamma$ by $N$ on both sides, we obtain $$\CJ(N\lambda + \rho, N \mu + \rho ; N \gamma) = b(N \gamma) \, * \, \Bigg( \sum_{\substack{\nu \in (\lambda + \mu) + N^{-1}Q \\ \ \ \cap \ \mathcal{C}_+}} C^{N\nu}_{N\lambda \, N\mu} \sum_{w \in W} \epsilon(w) \delta_{w(\nu + \rho/N)} \Bigg).$$ The homogeneity of $\CJ$ then gives (\ref{eqn:J-stretched}).
\end{proof}

The following immediate corollary is a distributional version of the semiclassical limit (\ref{eqn:C-J-sc}).  Let $d\gamma$ denote Lebesgue measure on $\tot$ and let $\implies$ denote weak convergence of finite signed measures.

\begin{corollary}
\label{cor:J-semiclassical}
For $\lambda, \mu$ dominant weights of $\gog$,
\begin{equation}
\label{eqn:J-semiclassical}
\frac{1}{N^{|\Phi^+|}} \sum_{\substack{\nu \in (\lambda + \mu) + N^{-1}Q \\ \ \ \cap \ \mathcal{C}_+}}  C^{N\nu}_{N\lambda\, N\mu} \sum_{w \in W} \epsilon(w)\, \delta_{w(\nu)} \implies \CJ(\lambda, \mu ; \gamma)\, d\gamma.
\end{equation}
\end{corollary}

\begin{proof}
Letting $N \to \infty$ in (\ref{eqn:J-stretched}) and recalling that $d = |\Phi^+| - r$, we obtain
\begin{align*}
\CJ(\lambda, \mu; \gamma) \ &= \ \lim_{N \to \infty} \frac{1}{N^d} b(N \gamma) \, * \, \Bigg( \sum_{\substack{\nu \in (\lambda + \mu) + N^{-1}Q \\ \ \ \cap \ \mathcal{C}_+}}  C^{N\nu}_{N\lambda\, N\mu} \sum_{w \in W} \epsilon(w) \, \delta_{w(\nu + \rho/N)} \Bigg) \\
\ &= \ \lim_{N \to \infty} N^r b(N \gamma) \, * \, \Bigg( \frac{1}{N^{|\Phi^+|}} \sum_{\substack{\nu \in (\lambda + \mu) + N^{-1}Q \\ \ \ \cap \ \mathcal{C}_+}}  C^{N\nu}_{N\lambda\, N\mu} \sum_{w \in W} \epsilon(w) \, \delta_{w(\nu)} \Bigg)
\end{align*}
almost everywhere $d\gamma$.

Since $||b||_{L^1} = 1,$ the sequence of functions $N^r b(N \gamma)$ is an approximation to the identity as $N \to \infty$, which implies the desired result.
\end{proof}

\subsection{The Littlewood--Richardson polynomial}
\label{sec:polynomiality}
In this subsection we take $\gog = \mathfrak{su}(n)$.  For a compatible triple $(\lambda, \mu, \nu)$ of dominant weights of $\mathfrak{su}(n)$, the function $P_{\lambda \mu}^\nu(N) = C^{N\nu}_{N\lambda \, N\mu}$ is in fact a polynomial in $N$, called the {\it Littlewood--Richardson polynomial}, rather than merely a quasipolynomial.  This polynomiality property of the stretched Littlewood--Richardson coefficients was proven by Rassart \cite{Rass} using vector partition functions and separately by Derksen and Weyman \cite{DW} using semi-invariants of quivers.  Both \cite{DW} and \cite{Rass} mention a third, unpublished proof by Knutson using ideas from symplectic geometry.  The symplectic geometry proof follows from the ``quantization commutes with reduction'' theorem in geometric quantization \cite{Kn}; see \cite{GS, MS, VQR} for background on this theorem.  These proofs confirmed part of a conjecture by King, Tollu and Toumazet \cite{KTT}, who later studied the degrees and factorization properties of the Littlewood--Richardson polynomials in \cite{KTT2}.  Here we use box spline deconvolution to give yet another proof of the polynomiality of $P_{\lambda \mu}^\nu$.

We essentially follow Rassart's recipe with different ingredients.  In particular, like Rassart, we take as given the quasi-polynomiality of $P_{\lambda \mu}^\nu(N)$, which follows already from the construction of Berenstein--Zelevinsky polyopes.  In a sense therefore we only give a novel approach to the final step required to show polynomiality, since we do not offer a new proof of quasi-polynomiality.  Nonetheless, this proof illustrates a number of ideas that we will use again later, in particular the relationship between the domains of polynomiality of $C_{\lambda \mu}^\nu$ and those of $\CJ$, which will be important in the proof of Theorem \ref{thm:fin-dif-inversion}.

Before giving the proof, we first recall some definitions and briefly sketch Rassart's argument.  For our purposes, a {\it hyperplane arrangement} $\mathcal{HA}$ in a Euclidean space $V$ is a finite union of (affine) hyperplanes.  The {\it chamber complex} associated to $\mathcal{HA}$ is a partition of $V$, consisting of all connected components of $V \setminus \mathcal{HA}$ together with the common refinement of the hyperplanes.  Each set in this partition is called a {\it chamber}.  If $\mathcal{HA}$ is {\it centered}, meaning that all of the hyperplanes pass through the origin, then the chamber complex partitions $V$ into convex polyhedral cones.

By \cite{BZ} the multiplicity $C_{\lambda \mu}^\nu$ equals the number of integer points in the polytope $H_{\lambda \mu}^\nu$.  By standard results about vector partition functions (see e.g. \cite{BSt}), this implies that there is a minimal centered hyperplane arrangement $\mathcal{LR}_n \subset \tot^3$, called the {\it Littlewood--Richardson arrangement}, such that $C_{\lambda \mu}^\nu$ is a quasi-polynomial function on each chamber of the associated chamber complex.  In particular, $C_{\lambda \mu}^\nu$ is quasipolynomial on rays starting at the origin, so that $P_{\lambda \mu}^\nu(N)$ is a quasipolynomial function of $N$.

Rassart defines the {\it Steinberg arrangement} $\mathcal{SA}_n$ to be the union of all affine hyperplanes of the form
\begin{equation} \label{eqn:SA-hyperplanes}
\langle \sigma(\alpha + \rho) + \tau(\beta + \rho) - (\gamma + 2\rho), \theta(\omega_j) \rangle = 0, \qquad \alpha, \beta, \gamma \in \tot, 
\end{equation}
for $\sigma, \tau, \theta \in W = S_n$ (the symmetric group) and $\omega_j$ a fundamental weight.  By analyzing the Kostant--Steinberg formula \cite{St} for the multiplicities $C_{\lambda \mu}^\nu$ (see (\ref{eqn:KS}) below), he concludes that $C_{\lambda \mu}^\nu$ is in fact polynomial in $(\lambda, \mu, \nu)$ on each connected component of $\tot^3 \setminus \mathcal{SA}_n$.  He then shows that each cone in the chamber complex of $\mathcal{LR}_n$ contains an arbitrarily large ball lying in some component of $\tot^3 \setminus \mathcal{SA}_n$, implying that the quasipolynomial determining $C_{\lambda \mu}^\nu$ on each cone must in fact coincide with one of the polynomials determining $C_{\lambda \mu}^\nu$ on $\tot^3 \setminus \mathcal{SA}_n$.

The proof below is similar in structure to Rassart's argument but uses different techniques at each step. Proposition \ref{prop:inversion-reworked} plays the role of the Kostant--Steinberg formula, while a different hyperplane arrangement, which we call the $\rho$-shifted Duistermaat--Heckman arrangement, plays the role of $\mathcal{SA}_n$.

\begin{theorem}[Rassart, Derksen--Weyman] \label{thm:polynomiality}
For a fixed compatible triple $(\lambda, \mu, \nu)$ of dominant weights of $\mathfrak{su}(n)$, $P_{\lambda \mu}^\nu(N)$ is a polynomial in $N$ with degree at most $\frac{1}{2}(n-1)(n-2).$
\end{theorem}

\begin{proof}
We identify $\mathfrak{su}(n)$ with the space of $n$-by-$n$ traceless anti-Hermitian matrices, and $\tot$ with the space of traceless diagonal matrices with imaginary entries.  By standard techniques from the theory of Duistermaat--Heckman measures~\cite{DH}, one can show (see e.g. \cite[sect. 2.1]{CMZ} for an explicit calculation) that the boundaries of the domains of polynomiality of $\CJ$ are contained in the hyperplanes
\begin{equation} \label{eqn:DH-arrangement}
\sum_{i \in I} \alpha_i + \sum_{j \in J} \beta_j - \sum_{k \in K} \gamma_k = 0, \qquad \alpha, \beta, \gamma \in \tot,
\end{equation}
where $I, J, K \subset \{1, \hdots, n\}$ with $|I| = |J| = |K|$, and $\alpha = i\, \mathrm{diag}(\alpha_1, \hdots, \alpha_n)$, etc. We call the hyperplane arrangement (\ref{eqn:DH-arrangement}) the {\it Duistermaat--Heckman arrangement} $\mathcal{DH}_n$.

In our chosen parametrization of $\tot$, we find $\rho = \frac{i}{2} \mathrm{diag}(n - 2j + 1)_{j = 1}^n$.  Applying the Weyl shift to the arguments in (\ref{eqn:DH-arrangement}) we find that as a function of $(\alpha, \beta, \gamma) \in \tot^3$, the regions of polynomiality of $\CJ(\alpha', \beta' ; \gamma')$ are cut out by the {\it $\rho$-shifted Duistermaat--Heckman arrangement} $\mathcal{DH}_n^\rho$ consisting of the hyperplanes
\begin{equation} \label{eqn:shifted-DH-arrangement}
\sum_{i \in I} (\alpha_i - i) + \sum_{j \in J} (\beta_j - j) - \sum_{k \in K} (\gamma_k - k) =  - \frac{1}{2}(n+1) |I|,
\end{equation}
where again $|I| = |J| = |K|$.  The arrangement $\mathcal{DH}_n^\rho$ is not centered, but its hyperplanes do all pass through the point $-(\rho, \rho, \rho)$, so that its chamber complex consists of convex polyhedral cones with this point as their common apex.

Define $P(\alpha, \beta, \gamma) := \lim_{t \to 0^+} \hat A(\Phi^+) \CJ(\alpha', \beta' ; \gamma' + t\eta)$, with $\eta$ defined as in Theorem \ref{thm:bs-deconv}.  By Corollary \ref{cor:inversion}, for $(\lambda, \mu, \nu)$ a compatible triple of dominant weights, we have $C_{\lambda \mu}^\nu = P(\lambda, \mu, \nu)$.  For $(\alpha, \beta, \gamma)$ not lying on one of the hyperplanes (\ref{eqn:shifted-DH-arrangement}), $P(\alpha, \beta, \gamma)$ can be represented as a finite linear combination of $\CJ(\alpha', \beta' ; \gamma')$ and its derivatives.  If $(\alpha, \beta, \gamma)$ does lie on one of the hyperplanes (\ref{eqn:shifted-DH-arrangement}), then taking the limit in (\ref{eqn:inversion-reworked}) we find that $P(\alpha, \beta, \gamma)$ can be represented as a finite linear combination of $\phi$ and its derivatives, where $\phi(\alpha, \beta, \gamma)$ is the local polynomial expression of $\CJ(\alpha', \beta' ; \gamma')$ on one of the adjacent connected components of $\tot^3 \setminus \mathcal{DH}_n^\rho$, as a function of the {\it unshifted} triple $(\alpha, \beta, \gamma)$.

On every connected component of $\tot^3 \setminus \mathcal{DH}_n^\rho$, either $\CJ(\alpha', \beta' ; \gamma')$ vanishes or it is a polynomial of degree $|\Phi^+| - (n-1) = \frac{1}{2}(n-1)(n-2)$ in the variables $(\alpha, \beta, \gamma)$.  Moreover, by (\ref{eqn:Ahat-rewritten}) the zeroth-order term of $\hat A(\Phi^+)$ is 1.  Thus we find that on every chamber of $\mathcal{DH}_n^\rho$, $P(\alpha, \beta, \gamma)$ is a polynomial of degree at most $\frac{1}{2}(n-1)(n-2)$, and its degree is exactly $\frac{1}{2}(n-1)(n-2)$ except on chambers where $\CJ(\alpha', \beta'; \gamma')$ vanishes or which lie on the boundary of a chamber where $\CJ(\alpha', \beta'; \gamma')$ vanishes.

We have now shown the existence of a locally polynomial expression for $C_{\lambda \mu}^\nu$.  However, we cannot yet conclude the polynomiality of $P_{\lambda \mu}^\nu(N)$, because the arrangement $\mathcal{DH}_n^\rho$ is not centered, so that a ray starting at the origin may pass through multiple polynomial domains.  To finish the proof, we compare $\mathcal{DH}_n^\rho$ to the centered arrangement $\mathcal{LR}_n$, following a similar geometric intuition to \cite[Theorem 4.1]{Rass} but using a different argument.

Consider the union of hyperplane arrangements $\mathcal{LR}_n \cup \mathcal{DH}_n^\rho$.  The chamber complex of this arrangement is the common refinement of the chamber complexes of $\mathcal{LR}_n$ and $\mathcal{DH}_n^\rho$.  Let $C$ be any cone of the chamber complex of $\mathcal{LR}_n$.  Then $C$ is partitioned into finitely many convex polyhedral chambers of $\mathcal{LR}_n \cup \mathcal{DH}_n^\rho$, each of which is contained in a single chamber of $\mathcal{DH}_n^\rho$.  There must be at least one chamber $R$ of this partition that is unbounded with $\dim R = \dim C$.  This region $R$ lies in a single polynomial domain of the locally polynomial function $P$ defined above.  Thus we know that $C_{\lambda \mu}^\nu$ is expressed as a quasi-polynomial $q(\lambda, \mu, \nu)$ on $C$ and as a polynomial $P(\lambda, \mu, \nu)$ on $R$, so that at all lattice points $(\lambda, \mu, \nu)$ in $R$, $q(\lambda, \mu, \nu)$ and $P(\lambda, \mu, \nu)$ must agree.  It follows that on the entire cone $C$, $q$ is in fact a polynomial and equals the polynomial that expresses $P$ locally on $R$.

Finally, fix a compatible triple $(\lambda, \mu, \nu)$ of dominant weights.  The points $(N \lambda, N \mu, N \nu)$ lie on a ray starting at the origin and therefore lie in a single cone of $\mathcal{LR}_n$, so that $P_{\lambda \mu}^\nu(N) = q(N \lambda, N \mu, N \nu)$ for some polynomial $q$ as above.  Thus $P_{\lambda \mu}^\nu$ is a polynomial in $N$ with $\deg P_{\lambda \mu}^\nu \le \deg q \le \frac{1}{2}(n-1)(n-2)$, as desired.
\end{proof}
\section{Discrete convolution and deconvolution}
\label{sec:discrete}

In this section we consider a discrete analogue of Proposition \ref{prop:J-BS}, which yields several new results. First we use this observation to develop two different methods for computing the multiplicities $C_{\lambda \mu}^\nu$ from finitely many values of $\CJ$, which work irrespective of unimodularity.  In the process, we make some novel observations on the $R$-polynomial of $\gog$, a function introduced in \cite{CZ1} and further studied in \cite{Coq2, CMZ, ER}, which turns out to be closely related to the box spline $B_c[\Phi^+]$. Next we derive some combinatorial identities involving the box spline density, which we use to relate Proposition \ref{prop:J-BS} to the $\CJ$-LR relation (\ref{eqn:JLR1}) and to express the discrete convolution with $b$ on the root lattice in terms of a finite difference operator called the box spline Laplacian.  Finally we prove that for $\gog = \mathfrak{su}(n)$, ``typical'' Littlewood--Richardson coefficients can be computed as an explicit finite linear combination of values of $\CJ$.

Our starting point is the observation that if we restrict attention to $\gamma = \nu'$ with $\nu \in \lambda + \mu + Q$, then Proposition \ref{prop:J-BS} says the following:

\begin{corollary} \label{cor:J-disc-conv}
\begin{equation} \label{eqn:J-disc-conv}
\sum_{\nu \in \lambda + \mu + Q} \CJ(\lambda', \mu' ; \nu') \, \delta_{\nu'} =  \Bigg (\sum_{\tau \in Q} b(\tau) \, \delta_\tau \Bigg) * \sum_{\substack{\tau \in \lambda + \mu + Q \\ \ \ \cap \ \mathcal{C}_+}} C_{\lambda \mu}^\tau \sum_{w \in W} \epsilon(w)\, \delta_{w(\tau')}.
\end{equation}
\end{corollary}

\begin{proof} All measures appearing in (\ref{eqn:J-disc-conv}) are supported on the weight lattice $P \subset \tot$. The convolution of two arbitrary measures on $P$ may be written
\begin{equation} \label{eqn:disc-conv-formula} 
\Bigg( \sum_{\tau \in P} f_\tau \delta_\tau \Bigg) * \Bigg( \sum_{\tau \in P} g_\tau \delta_\tau \Bigg) = \sum_{\tau \in P} \sum_{\nu \in P} f_\nu g_{\tau - \nu} \, \delta_\tau,
\end{equation}
where $\{ f_\tau, g_\tau \}_{\tau \in P}$ are some coefficients.

From Proposition \ref{prop:J-BS} we have
\begin{equation} \label{eqn:J-conv-basic} \CJ(\lambda', \mu'; \nu') = \sum_{\substack{\tau \in \lambda + \mu + Q \\ \ \ \cap \ \mathcal{C}_+}} C_{\lambda \mu}^\tau \sum_{w \in W} \epsilon(w) b(\nu' - w(\tau')),\end{equation}
so that $$ \sum_{\nu \in \lambda + \mu + Q} \CJ(\lambda', \mu'; \nu') \, \delta_{\nu'} = \sum_{\nu \in \lambda + \mu + Q \ } \sum_{\substack{\tau \in \lambda + \mu + Q \\ \ \ \cap \ \mathcal{C}_+}} C_{\lambda \mu}^\tau \sum_{w \in W} \epsilon(w) b(\nu' - w(\tau')) \, \delta_{\nu'}.$$  Comparing the right-hand side above to (\ref{eqn:disc-conv-formula}) and observing that $\nu' - w(\tau')$ runs over $Q$ for each fixed $\nu$, we conclude (\ref{eqn:J-disc-conv}).
\end{proof}

Corollary \ref{cor:J-disc-conv} has the same form as Proposition \ref{prop:J-BS} but involves measures that we can think of as discrete approximations of $\CJ$ and $b$.  In section \ref{sec:b-idents} we show that this statement is actually equivalent to the $\CJ$-LR relation ({\ref{eqn:JLR1}). One consequence of Corollary \ref{cor:J-disc-conv}, which we show in the next subsection, is that tensor product multiplicities can be computed from the volume function in all cases, whether or not $\Phi^+$ is unimodular.  Moreover, for a given triple $(\lambda, \mu, \nu)$, only finitely many values of $\CJ(\lambda', \mu' ; \gamma)$ are required, and it is not necessary to compute derivatives of $\CJ$ as in Corollary \ref{cor:inversion}.

\subsection{Finite deconvolution on a lattice}
\label{sec:lattice-deconv}

In this subsection we give two methods for computing $C_{\lambda \mu}^\nu$ from $\CJ$: a constructive algorithm for computing the multiplicities algebraically, and a method based on Fourier analysis leading to an integral representation of the multiplicities.  These two methods are complementary, since the integral representation is easy to write down but may be difficult to evaluate analytically, while the algebraic procedure is straightforward to execute computationally but does not lead a priori to a clean formula.  We also make some new observations on the vanishing locus of the $R$-polynomial of $\gog$ studied in \cite{Coq2, CMZ, CZ1, ER}, whose definition we recall in Definition \ref{def:R-poly} below.

We start with the algebraic approach.  The key insight is that no information is lost by considering only the discrete measure in (\ref{eqn:disc-conv-formula}) rather than the full function $\CJ$.  This is immediate from the following lemma, which provides a general deconvolution algorithm for finitely supported functions on a lattice.  The result below must be well known, but the precise statement that we need was difficult to find in the literature, so we prove it here.

\begin{lemma} \label{lem:disc-deconv}
Let $\Lambda$ be a lattice, let $f: \Lambda \to \C$ and $g: \Lambda \to \C$ be finitely supported functions, and suppose that $f$ is not uniformly zero.  Let $h := f * g$ be their convolution, i.e.,
\begin{equation} \label{eqn:disc-conv-def}
h(\lambda) := \sum_{\mu \in \Lambda} f(\mu)g(\lambda - \mu), \qquad \lambda \in \Lambda.
\end{equation}
Then $g$ can be computed from $f$ and $h$ using finitely many arithmetic operations.
\end{lemma}

\begin{proof}
The algorithm follows a recursive procedure.  At each iteration we solve a linear programming problem that allows us to determine one additional nonzero value of $g$.  When all nonzero values of $g$ have been computed, the procedure terminates.

First observe that $||h||_{\ell^1} = ||f||_{\ell^1} ||g||_{\ell^1}$.  Therefore if $h$ is uniformly zero then $g$ must be uniformly zero, and we are done.  Suppose therefore that $h$ is not uniformly zero.

Let $V := \Lambda \otimes \R$ be the real span of $\Lambda$.  Given a function $\varphi: \Lambda \to \C$, let $\mathrm{supp}(\varphi) \subset \Lambda$ denote its support and $\mathrm{Conv}(\varphi) \subset V$ denote the convex hull of $\mathrm{supp}(\varphi)$.  Since $\mathrm{supp}(h)$ is finite and non-empty, $\mathrm{Conv}(h)$ has at least one vertex.  Choose such a vertex $\tau$; we can then find a linear functional $\ell \in V^*$ such that $\tau$ maximizes $\ell$ uniquely over $\mathrm{Conv}(h)$.  From this unique maximization property and the fact that $\mathrm{supp}(h)$ is contained in the Minkowski sum $\mathrm{supp}(f) + \mathrm{supp}(g)$, it follows that there is a unique way of writing $\tau = \tau_1 + \tau_2$ with $\tau_1 \in \mathrm{supp}(f)$ and $\tau_2 \in \mathrm{supp}(g)$.  Moreover, $\tau_1$ and $\tau_2$ are the unique maximizers of $\ell$ over $\mathrm{Conv}(f)$ and $\mathrm{Conv}(g)$ respectively, so we can compute $\tau_1$ by linear programming and $\tau_2$ as $\tau - \tau_1$.

By the definition (\ref{eqn:disc-conv-def}), we have $$h(\tau) = \sum_{\mu \in \Lambda} f(\mu)g(\tau - \mu) = f(\tau_1)g(\tau_2) + \sum_{\mu \ne \tau_1} f(\mu)g(\tau - \mu).$$  But we have just argued that $\tau_1$ is the unique $\mu \in \mathrm{supp}(f)$ such that $\tau - \mu \in \mathrm{supp}(g)$, so the last sum above vanishes, giving $h(\tau) = f(\tau_1)g(\tau_2)$.  We conclude that $g(\tau_2) = h(\tau) / f(\tau_1)$.

Having determined $g(\tau_2)$, we can remove $\tau_2$ from $\mathrm{supp}(g)$ and repeat the procedure.  That is, we apply the same algorithm to the function $$h'(\lambda) := h(\lambda) - g(\tau_2)f(\lambda - \tau_2) = (f * g')(\lambda),$$ where $g'(\lambda) = g(\lambda)$ for $\lambda \ne \tau_2$ and $g'(\tau_2) = 0$.  Since $\mathrm{supp}(g)$ was assumed finite, after $|\mathrm{supp}(g)|$ iterations we obtain the zero function, at which point we know that all values of $g$ have been determined and the algorithm terminates.
\end{proof}

We can identify each $f \in \ell^1(P)$ with the complex measure $\sum_{\tau \in P} f(\tau) \, \delta_\tau$, so that convolution of measures as in (\ref{eqn:disc-conv-formula}) is equivalent to convolution of functions as in (\ref{eqn:disc-conv-def}). Recognizing that Corollary \ref{cor:J-disc-conv} then expresses the restriction of $\CJ(\lambda', \mu'; \gamma)$ to $\lambda + \mu + \rho + Q$ as a convolution of two finitely supported functions on $P$, we see that we have shown the following.

\begin{theorem} \label{thm:C-from-J-algo}
For any $\gog$ and any dominant weights $\lambda, \mu, \nu$, the multiplicity $C_{\lambda \mu}^\nu$ can be computed constructively as an arithmetic expression in finitely many values of $\CJ(\lambda', \mu' ; \gamma)$ and $b(\gamma)$.
\end{theorem}

Although Lemma \ref{lem:disc-deconv} does provide a constructive method for computing $C_{\lambda \mu}^\nu$ in terms of $\CJ$, the procedure may be quite lengthy, and the form of the expression thus obtained may depend on $\lambda$ and $\mu$ in a complicated way.  It will not necessarily be the case that the algorithm above leads to a compact or transparent identity relating the volume function to the multiplicities.  However, for many triples of highest weights of $\mathfrak{su}(n)$, a simple expression does indeed exist for $C_{\lambda \mu}^\nu$ as a linear combination of values of $\CJ$.  We return to this topic in section \ref{sec:fd-inversion} below.

It should also be noted that Theorem \ref{thm:C-from-J-algo} depends crucially on the fact that $\CJ(\lambda', \mu' ; \gamma)$ is compactly supported.  On the other hand, even though Theorem \ref{thm:bs-deconv} and its analogue for the non-unimodular case in \cite{DV} require more information in a sense (derivatives of $\CJ$, as well as additional piecewise polynomial functions when $\Phi^+$ is not unimodular), they also apply in a more general setting, as they do not assume compact support for the discrete distribution appearing in the convolution with $b$.

Next we give the Fourier-analytic approach.  We will derive an inversion formula involving a function called the {\it $R$-polynomial} of $\gog$, introduced in \cite{CZ1} and further studied in \cite{Coq2, CMZ, ER}.
\begin{definition} \label{def:R-poly} \normalfont
The {\it $R$-polynomial} of $\gog$ is the function
\begin{equation} \label{eqn:R-poly-def}
R(x) := \sum_{\eta \in 2\pi P^\vee} j_\gog^{1/2}(x + \eta), \qquad x \in \tot,
\end{equation}
where $P^\vee \subset \tot$ is the coweight lattice and $j_\gog^{1/2}$ is defined by (\ref{eqn:BS-FT}).
\remdone
\end{definition}

The $R$-polynomial can be regarded either as a trigonometric polynomial on $\tot$ or as a genuine polynomial on the torus $\exp(\tot) \cong \tot / 2\pi P^\vee$. As a trigonometric polynomial on $\tot$, its coefficients are given by the values of $b$ at points of the root lattice:
\begin{equation} \label{eqn:R-from-b}
R(x) = \sum_{\tau \in Q} b(\tau) \cos(\langle \tau, x \rangle), \qquad x \in \tot.
\end{equation}
This follows immediately from the definition (\ref{eqn:R-poly-def}). Since $P^\vee$ is the dual lattice of $Q$ and $j_\gog^{1/2} = \mathscr{F}[b]$, the Poisson summation formula gives
\begin{equation} \label{eqn:b-R-FT}
\sum_{\eta \in 2\pi P^\vee} j_\gog^{1/2}(x + \eta) = \sum_{\tau \in Q} b(\tau) e^{i \langle \tau, x \rangle},
\end{equation}
and (\ref{eqn:R-from-b}) then follows from the symmetry $b(\tau) = b(-\tau)$.  Note that the sum in (\ref{eqn:R-from-b}) is actually finite, since $b$ is compactly supported.

Another expression for the $R$-polynomial was proved by Etingof and Rains in \cite{ER}:
\begin{equation} \label{eqn:R-ER}
R(x) = \sum_{\kappa \in K} r_\kappa \chi_\kappa(e^x),
\end{equation}
where $K \subset Q$ consists of all dominant elements of the root lattice that lie on the interior of the convex hull of the Weyl orbit of $\rho$, and $r_\kappa := \CJ(\rho, \rho, \kappa')$.

Recall that the dual of the weight lattice $P$ is the coroot lattice $Q^\vee$.  Let $|Q^\vee|$ denote the volume of a fundamental domain of $Q^\vee$.  We have the following integral representation for tensor product multiplicities:
\begin{theorem} \label{thm:C-from-J-Fourier}
For any dominant weights $\lambda, \mu, \nu$ of $\gog$,
\begin{equation} \label{eqn:C-from-J-Fourier}
C_{\lambda \mu}^\nu = \frac{1}{(2\pi)^r|Q^\vee|} \int_{\tot / 2\pi Q^\vee} \frac{1}{R(x)} \sum_{\tau \in \lambda + \mu + Q} \CJ(\lambda', \mu' ; \tau')\, e^{i \langle \tau - \nu, x \rangle} \, dx.
\end{equation}
\end{theorem}
\begin{proof}
We think of $P$ as the character group of a torus $T \cong \tot / 2\pi Q^\vee$.  Taking Fourier transforms on both sides of (\ref{eqn:J-disc-conv}), we obtain an equality of functions \mbox{on $T$:}
\begin{align*}
\sum_{\tau \in \lambda + \mu + Q} \CJ(\lambda', \mu' ; \tau') \, e^{i\langle \tau', x \rangle} &=  \Bigg (\sum_{\tau \in Q} b(\tau)\, e^{i\langle \tau, x \rangle} \Bigg) \Bigg( \sum_{\substack{\tau \in \lambda + \mu + Q \\ \ \ \cap \ \mathcal{C}_+}} C_{\lambda \mu}^\tau \sum_{w \in W} \epsilon(w)\, e^{i\langle w(\tau'), x \rangle} \Bigg) \\
&= R(x) \Bigg( \sum_{\substack{\tau \in \lambda + \mu + Q \\ \ \ \cap \ \mathcal{C}_+}} C_{\lambda \mu}^\tau \sum_{w \in W} \epsilon(w)\, e^{i\langle w(\tau'), x \rangle} \Bigg).
\end{align*}
Since $R(x)$ is a nonzero trigonometric polynomial, it can vanish at most on a set of measure zero.  Therefore the following holds almost everywhere with respect to Haar measure on $T$:
\begin{equation} \label{eqn:C-integrand} \sum_{\substack{\tau \in \lambda + \mu + Q \\ \ \ \cap \ \mathcal{C}_+}} C_{\lambda \mu}^\tau \sum_{w \in W} \epsilon(w)\, e^{i\langle w(\tau'), x \rangle} = \frac{1}{R(x)}\sum_{\tau \in \lambda + \mu + Q} \CJ(\lambda', \mu' ; \tau')\, e^{i \langle \tau', x \rangle}.\end{equation}
The left-hand side of this equation is a finite sum of bounded functions, so it is integrable on $T$.  Therefore we can recover the coefficient $C_{\lambda \mu}^\nu$ by integrating the right-hand side against $e^{-i \langle \nu', x \rangle}$ with respect to the normalized Haar measure $(2\pi)^{-r} |Q^\vee|^{-1}dx$, where $dx$ is Lebesgue measure on a fundamental domain of $2 \pi Q^\vee$ in $\tot$.  This completes the proof.
\end{proof}

Even though $R(x)$ may have zeros, the integral in (\ref{eqn:C-from-J-Fourier}) is always absolutely convergent.  Indeed, (\ref{eqn:C-integrand}) implies that the integrand is nonsingular, so any zeros of $R(x)$ in the denominator must be canceled in the numerator by zeros of the sum $\sum_{\tau \in \lambda + \mu + Q} \CJ(\lambda', \mu' ; \tau')\, e^{i \langle \tau - \nu, x \rangle}$.  We cannot necessarily pull the summation out of the integral, since without these cancelations the individual terms $R(x)^{-1} \CJ(\lambda', \mu' ; \tau')\, e^{i \langle \tau - \nu, x \rangle}$ may fail to be integrable.  If we could pull the sum out, however, then we could rewrite (\ref{eqn:C-from-J-Fourier}) in the much cleaner form:
\begin{equation} \label{eqn:C-from-conv}
C_{\lambda \mu}^\nu = \sum_{\tau \in \lambda + \mu + Q} \CJ(\lambda', \mu'; \tau') \, c(\nu - \tau),
\end{equation}
where
\begin{equation} \label{eqn:c-kernel}
c(\tau) := \frac{1}{(2\pi)^r|Q^\vee|} \int_{\tot / 2\pi Q^\vee} \frac{e^{- i \langle \tau, x \rangle}}{R(x)} \, dx, \qquad \tau \in Q.
\end{equation}
In fact the function $c$ would then provide a universal deconvolution kernel for the discrete convolution with $b$, such that $$\Bigg( \sum_{\tau \in Q} c(\tau) \delta_\tau \Bigg) * \Bigg( \sum_{\tau \in Q} b(\tau) \delta_\tau \Bigg) = \delta_0.$$

This would be convenient, but in general it is impossible.  The integral in (\ref{eqn:c-kernel}) is absolutely convergent if and only if $R(x)$ is nowhere vanishing, and the following proposition shows that this fails in the non-unimodular case.

\begin{proposition} \label{prop:bFT-vanish}
If $\Phi^+$ is not unimodular, then $R(x) = 0$ for some $x \in \tot$.
\end{proposition}

\begin{proof}
By (\ref{eqn:b-R-FT}), the $R$-polynomial is the Fourier transform of $b|_Q$.  Therefore, by Wiener's Tauberian theorem \cite{NW}, $R(x)$ is nonvanishing if and only if the translates of $b|_Q$ are dense in $\ell^1(Q)$. Since $\ell^1(Q)^* \cong \ell^\infty(Q)$, this in turn is equivalent to the statement that there is no nonzero $a \in \ell^\infty(Q)$ such that $$\sum_{\tau \in Q} a(\tau) b(\nu - \tau) = 0$$ for all $\nu \in Q$.  We know from Theorem \ref{thm:T-injective} however that if $\Phi^+$ is not unimodular then such an $a \in \ell^\infty(Q)$ does exist, so $R(x)$ must vanish somewhere \mbox{on $\tot$.}
\end{proof}

Therefore in the non-unimodular case, (\ref{eqn:c-kernel}) is not well defined a priori, and we cannot expect a further simplification of (\ref{eqn:C-from-J-Fourier}).

The case of $\mathfrak{su}(n)$ is less clear, however.  For $n=2$ or 3, $R(x) \equiv 1$, giving $c(0) = 1$ and $c(\tau) = 0$ for $\tau \ne 0$, which recovers (\ref{Ahat-su2su3}).  For larger $n$ it quickly becomes difficult to study the $R$-polynomial analytically, due to an explosion in the number of terms.  Nonetheless, numerical investigations up to $n=7$, using formulae for $R(x)$ derived in \cite{Coq2}, suggest that the $R$-polynomial remains strictly positive in these cases.  We venture the following conjecture.

\begin{conjecture} \label{conj:cos-nonvanish}
{\it 
If $\Phi^+$ is unimodular, then $R(x) > 0$ for all $x \in \tot.$
}
\end{conjecture}

\begin{remark} \label{rem:nonvanish-condition} \normalfont
The proof of Proposition \ref{prop:bFT-vanish} amounts to the observation that $R(x)$ is nonvanishing if and only if 
$\ker (\mathrm{Res}_Q \circ \mathcal{T}) \cap \ell^\infty(Q) = \{0\},$
where $\mathcal{T}$ is the operator defined in (\ref{eqn:T-def}) and $\mathrm{Res}_Q$ is restriction to $Q$.  In the unimodular case Theorem \ref{thm:T-injective} only tells us that $\ker(\mathcal{T}) = \{0\}$, which is not immediately enough to conclude that $\ker (\mathrm{Res}_Q \circ \mathcal{T}) = \{0\}$. \remdone
\end{remark}

\subsection{Identities for the box spline and the $R$-polynomial}
\label{sec:b-idents}

We next record some identities regarding the coefficients $r_\kappa$ and the values of the box spline density $b$ at points of the root lattice $Q$. One implication of these results is that Corollary \ref{cor:J-disc-conv} is actually equivalent to the $\CJ$-LR relation (\ref{eqn:JLR1}).  In the following subsection we will use these identities to express the discrete convolution with $b$ in terms of finite difference operators, which will be a key ingredient in the proof of Theorem \ref{thm:fin-dif-inversion} in Section \ref{sec:fd-inversion} below.

First, it is a well-known fact that the lattice translates of a box spline form a partition of unity; see \cite[sect. 2.1]{PBP}.  This immediately implies:
\begin{proposition} \label{prop:b-sum-1}
\begin{equation} \label{eqn:b-sum-1} R(0) = \sum_{\tau \in Q} b(\tau) = 1. \end{equation}
\end{proposition}

Second, we can express the values $b(\tau)$ in terms of the coefficients $r_\kappa$ and vice versa.  For $\lambda, \mu$ dominant weights of $\gog$, let $\mathrm{mult}_\lambda(\mu)$ represent the multiplicity of the weight $\mu$ in the irreducible representation $V_\lambda$.
\begin{proposition} \label{prop:b-on-Q}
\begin{eqnarray}
\label{eqn:b-on-Q}
b(\tau) &=& \sum_{\kappa \in K} r_\kappa \mathrm{mult}_\kappa(\tau), \qquad \qquad \ \ \tau \in Q, \\
\label{eqn:r-b-relation}
r_\kappa &=& \sum_{w \in W} \epsilon(w)\, b(\kappa' - w(\rho)), \qquad \kappa \in K.
\end{eqnarray}
\end{proposition}
\begin{proof}
Since $K \subset Q$, we can write $\chi_\kappa(e^x) = \sum_{\tau \in Q} \mathrm{mult}_\kappa(\tau) e^{i \langle \tau, x \rangle}$.  Then (\ref{eqn:b-on-Q}) follows by equating (\ref{eqn:R-from-b}) and (\ref{eqn:R-ER}).  The relation (\ref{eqn:r-b-relation}) is immediate from Proposition \ref{prop:J-BS} and the definition $r_\kappa = \CJ(\rho, \rho, \kappa').$
\end{proof}

Comparing (\ref{eqn:b-on-Q}) and (\ref{eqn:r-b-relation}), we find the following further relations:

\begin{corollary} \label{cor:r-b-rels}
\begin{eqnarray}
\label{eqn:r-b-rels_b}
b(\tau) &=& \sum_{w \in W} \epsilon(w) \sum_{\kappa \in K} b(\kappa' - w(\rho)) \, \mult_\kappa(\tau), \quad \tau \in Q, \\
\label{eqn:r-b-rels_r}
r_\kappa &=& \sum_{w \in W} \epsilon(w) \sum_{\xi \in K} r_\xi \, \mult_\xi(\kappa' - w(\rho)), \qquad \kappa \in K.
\end{eqnarray}
\end{corollary}

\begin{example} \label{ex:BS-idents} \normalfont
For $\mathfrak{su}(2)$ and $\mathfrak{su}(3)$, we have $K = Q \cap \mathrm{supp}(b) = \{ 0 \}$ and $b(0) = 1$.

For $\mathfrak{su}(4)$ we have $K = \{ 0, \alpha_+ \}$, where $\alpha_+$ is the single root that lies in $\mathcal{C}_+$, while $Q \cap \mathrm{supp}(b)$ consists of $0$ and the 12 roots. We find $r_0 = 9/24,$ $r_{\alpha_+} = 1/24$, and $\mult_0(0) = 1$, $\mult_0(\alpha_+) = 0$, $\mult_{\alpha_+}(\alpha_+) = 1$, $\mult_{\alpha_+}(0) = 3$.  From Proposition \ref{prop:b-on-Q} we thus obtain $b(0) = 1/2$ and $b(\alpha) = 1/24$ for $\alpha$ any root, confirming that $$\sum_{\tau \in Q} b(\tau) = \frac{1}{2} + 12 \cdot \frac{1}{24} = 1.$$

For $\mathfrak{su}(5)$, the $\tau \in Q$ with $b(\tau) \not = 0$ are $0$, the 20 roots, and the Weyl orbit consisting of the 30 points of the form $\alpha + \beta$ where $\alpha$ and $\beta$ are any two orthogonal roots.  Computing as above we find $b(0) = 1/4$, $b(\alpha) = 1/30$ for $\alpha$ any root, and $b(\alpha + \beta) = 1/360$ for $\beta$ any root orthogonal to $\alpha$, again giving $\sum_{\tau \in Q} b(\tau) = 1.$

For $\mathfrak{so}(5)$, there are two simple roots: a long root $\alpha_1$ and a short root $\alpha_2$.  We have $K = \{0, \alpha_1 + \alpha_2 \}$, with $r_0 = 3/8$ and $r_{\alpha_1 + \alpha_2} = 1/8$, and $Q \cap \mathrm{supp}(b)$ consists of 0 and the 4 short roots.  We find $\mult_0(0) = 1$, $\mult_0(\alpha_2) = 0$, $\mult_{\alpha_1 + \alpha_2}(0) = \mult_{\alpha_1 + \alpha_2}(\alpha_2) = 1$.  All together, this gives $b(0) = 1/2$ and $b(\alpha) = 1/8$ for $\alpha$ any short root, giving $\sum_{\tau \in Q} b(\tau) = 1$ as expected. \remdone
\end{example}

We now recall a couple of classical multiplicity formulae that we will use below.  Let $\mathrm{Part}: Q \to \N$ be the Kostant partition function, which counts the number of distinct ways that an element of the root lattice can be decomposed as a positive integer linear combination of the positive roots. Then we have the Kostant multiplicity formula \cite{Kost}:
\begin{equation} \label{eqn:kostant-mult}
\mathrm{mult}_\lambda(\mu) = \sum_{w \in W} \epsilon(w)\, \mathrm{Part}(w(\lambda') - \mu').
\end{equation}
In particular, (\ref{eqn:kostant-mult}) implies that $\mathrm{mult}_\lambda(\mu) = 0$ whenever $\mu$ lies outside the convex hull of the Weyl orbit of $\lambda$.

We also have the Kostant--Steinberg formula for $C_{\lambda \mu}^\nu$ \cite{St}:
\begin{equation} \label{eqn:KS}
C_{\lambda \mu}^\nu = \sum_{w, w' \in W} \epsilon(ww')\, \mathrm{Part}(w(\lambda') + w'(\mu') - \nu' - \rho).
\end{equation}
Putting (\ref{eqn:kostant-mult}) and (\ref{eqn:KS}) together, we get:
\begin{equation} \label{eqn:C-from-mult}
C_{\lambda \mu}^\nu = \sum_{w \in W} \epsilon(w)\, \mathrm{mult}_\lambda(w(\mu') - \nu').
\end{equation}

To end this subsection, we observe that from (\ref{eqn:C-from-mult}) and Proposition \ref{prop:b-on-Q}, we obtain:
\begin{equation} \label{eqn:b-KS}
\sum_{w \in W} \epsilon(w)\, b(w(\tau') - \nu') = \sum_{\kappa \in K} r_\kappa C_{\tau \kappa}^\nu.
\end{equation}
Comparing this with the discrete convolution formula (\ref{eqn:J-disc-conv}) for $\CJ$, we recover the $\CJ$-LR relation (\ref{eqn:JLR1}), and vice versa.  We have thus shown that (\ref{eqn:J-disc-conv}) and (\ref{eqn:JLR1}) are equivalent.

\subsection{The box spline Laplacian}
\label{sec:finite-difference}

In this subsection we introduce a finite difference operator called the box spline Laplacian, which allows us to give a convenient new representation of the discrete convolution with $b$.  This leads in turn to a reformulation of the convolution identity in Corollary \ref{cor:J-disc-conv} and to a representation of $\hat A(\Phi^+)$ as a finite difference operator on the space $D(\Phi^+)$ defined in (\ref{eqn:DMspace}), both of which will be useful in Section \ref{sec:fd-inversion} below.

\begin{definition} \label{def:BS-Laplacian} \normalfont
For $\tau$ in the weight lattice, let $\Delta_\tau$ and $\nabla_\tau$ denote respectively the forwards and backwards finite difference operators in the direction of $\tau$:
\begin{eqnarray*} \label{eqn:fin-diff-def}
\Delta_\tau f(x) &:=& f(x+\tau) - f(x), \\
\nabla_\tau f(x) &:=& f(x) - f(x-\tau), \qquad f: \tot \to \C.
\end{eqnarray*}
Define the {\it box spline Laplacian} $\CD$ by
\begin{equation} \label{eqn:D-def}
\CD := \sum_{\tau \in Q} b(\tau) \nabla_\tau \Delta_\tau.
\end{equation}
The term $b(0)\nabla_0$, while formally included, contributes nothing.  We will sometimes consider $\Delta_\tau$, $\nabla_\tau$ and $\CD$ as operators on $\tot^3$, in which case we will always assume that they act in the third argument, so that e.g.: $$\nabla_\tau \CJ(\lambda', \mu'; \nu') = \CJ(\lambda', \mu'; \nu') - \CJ(\lambda', \mu'; \nu' - \tau).$$ \remdone
\end{definition}

The significance of $\CD$ is that we can use it to represent the discrete convolution with $b$ as a finite difference operator.

\begin{proposition} \label{prop:b-conv-D-rep}
Let $m : Q \to \C$ and let $$m_b(\nu) := \sum_{\tau \in Q} b(\tau) m(\nu - \tau), \qquad \nu \in Q$$ be its discrete convolution with $b$ on $Q$.  Then
\begin{equation} \label{eqn:b-conv-D-rep}
m_b(\nu) = \left( 1 + \frac{1}{2} \CD \right) m(\nu), \qquad \nu \in Q.
\end{equation}
\end{proposition}

\begin{proof}
Proposition \ref{prop:b-sum-1} gives:
\begin{eqnarray}
\nonumber m_b(\nu) &=& b(0) m(\nu) + \sum_{\tau \not = 0} b(\tau) m(\nu - \tau) \\
\nonumber &=& \Big(1 - \sum_{\tau \not = 0} b(\tau) \Big) m(\nu) + \sum_{\tau \not = 0} b(\tau) m(\nu - \tau) \\
\nonumber &=& m(\nu) - \sum_{\tau \in Q} b(\tau) (m(\nu) - m(\nu - \tau)) \\
\nonumber &=& \Big(1 - \sum_{\tau \in Q} b(\tau) \nabla_\tau \Big) m(\nu).
\end{eqnarray}
By the reflection symmetry $b(\tau) = b(-\tau)$ and the fact that $\nabla_0 = 0$, we can rewrite this last expression to get
$$m_b(\nu) = \Big(1 - \frac{1}{2}\sum_{\tau \in Q} b(\tau) (\nabla_\tau + \nabla_{-\tau}) \Big) m(\nu).$$
The claim then follows from the observation that $$(\nabla_\tau + \nabla_{-\tau})m(\nu) = -m(\nu+\tau) + 2 m(\nu) - m(\nu-\tau) = - \nabla_\tau \Delta_\tau m(\nu).$$
\end{proof}

Proposition \ref{prop:b-conv-D-rep} leads to a useful reformulation of the discrete convolution identity in Corollary \ref{cor:J-disc-conv}. For any $\lambda \in P$, not necessarily dominant, let $\lambda_+ \in \mathcal{C}_+$ denote the unique dominant element in the Weyl orbit of $\lambda$, and let $\lambda_* := (\lambda')_+ - \rho$.  Define a function $\mathscr{C}$ on $P^3$ by
\begin{equation} \label{eqn:curly-C-def}
\mathscr{C}(\lambda, \mu, \nu) := C_{\lambda_* \mu_*}^{\nu_*},
\end{equation}
observing that for $\lambda, \mu, \nu$ all dominant, $\mathscr{C}(\lambda, \mu, \nu) = C_{\lambda \mu}^{\nu}.$  Combining Corollary \ref{cor:J-disc-conv} and Proposition \ref{prop:b-conv-D-rep}, we get:
\begin{corollary} \label{cor:JLR-laplacian}
For $(\lambda, \mu, \nu)$ a compatible triple of dominant weights of $\gog$, 
\begin{equation} \label{eqn:JLR-laplacian}
\CJ(\lambda', \mu' ; \nu') = \left( 1 + \frac{1}{2} \CD \right) \mathscr{C}(\lambda, \mu, \nu).
\end{equation}
\end{corollary}
Recall that in the expression above, $\CD$ acts in the third argument of $\mathscr{C}$.

\begin{example} \label{ex:JLR-laplacian} \normalfont
For $\mathfrak{su}(2)$ and $\mathfrak{su}(3)$, $\CD = 0$, so that (\ref{eqn:JLR-laplacian}) just reads $\CJ(\lambda', \mu' ; \nu') = C_{\lambda \mu}^\nu$, which we already know from (\ref{Ahat-su2su3}).

For $\mathfrak{su}(4)$, from the values of $b$ computed in Example \ref{ex:BS-idents}, we find $\CD = \frac{1}{12} \sum_{\alpha \in \Phi^+} \nabla_\alpha \Delta_\alpha$, and (\ref{eqn:JLR-laplacian}) reads
\begin{equation} \label{JLR-lap-su4}
\CJ(\lambda', \mu' ; \nu') =  \left( 1 + \frac{1}{24} \sum_{\alpha \in \Phi^+} \nabla_\alpha \Delta_\alpha \right) C_{\lambda \mu}^\nu,
\end{equation}
which was already observed in \cite[sect. 4.2.2]{CZ1}.

For $\mathfrak{su}(5)$, again using the values of $b$ computed in Example \ref{ex:BS-idents}, after some manipulation we find 
\begin{equation} \label{eqn:JLR-lap-su5}
\CD = \frac{1}{15} \sum_{\alpha \in \Phi^+} \Bigg( \nabla_\alpha \Delta_\alpha + \frac{1}{12} \sum_{\substack{\beta \in \Phi^+ \\ \langle \beta, \alpha \rangle = 0}} \big ( \nabla_{\alpha + \beta} \Delta_{\alpha+\beta} + \nabla_{\alpha - \beta} \Delta_{\alpha-\beta} \big ) \Bigg).
\end{equation}

For $\mathfrak{so}(5)$ we find $\CD = \frac{1}{4} ( \nabla_\alpha \Delta_\alpha + \nabla_\beta \Delta_\beta )$, where $\alpha$ and $\beta$ are the two short positive roots, which are orthonormal vectors in $\tot \cong \R^2$.  Thus $\CD$ is just $1/4$ times the usual discrete Laplacian in two dimensions. \remdone
\end{example}

\begin{remark} \label{rem:D-neumann} \normalfont
Proposition \ref{prop:b-conv-D-rep} might lead one to hope that the discrete convolution with $b$ could be inverted via the Neumann series
\begin{equation} \label{eqn:neumann}
\Big(1 + \frac{1}{2} \CD \Big)^{-1} = 1 - \frac{1}{2} \CD + \frac{1}{4} \CD^2 - \cdots,
\end{equation}
but in general this is not the case.  For (\ref{eqn:neumann}) to hold on any given Banach space, the series on the right-hand side must converge in the operator norm, which occurs if and only if the spectral radius of $\CD$ is strictly less than 2.  Except in the trivial cases of $\mathfrak{su}(2)$ and $\mathfrak{su}(3)$ where $\CD = 0$, one can show that this fails on many spaces of interest such as $\ell^p(Q)$, $1 \le p \le \infty$, as well as the space $c_0(Q)$ of functions on $Q$ that decay to zero at infinity.  However, (\ref{eqn:neumann}) {\it does} hold on the space $D(\Phi^+)$ defined in (\ref{eqn:DMspace}): $(1 + \frac{1}{2} \CD)$ is invertible on $D(\Phi^+)$ due to Theorem \ref{thm:A-T-inverse}, and since $D(\Phi^+)$ is a space of polynomials of degree $d = |\Phi^+| -r$, the following lemma shows that the Neumann series truncates after $\lfloor d/2 \rfloor$ terms. \remdone
\end{remark}

\begin{lemma} \label{lem:D-deg}
Let $p$ be a polynomial on $\tot$. If $\deg p \ge 2$ then $\deg \CD p \le \deg p - 2,$ and if $\deg p < 2$ then $\CD p = 0$.
\end{lemma}
\begin{proof}
This is immediate from the facts that $$\deg \nabla_\tau p \ = \ \deg \Delta_\tau p \ \le \ \max(\deg p - 1, 0)$$ and that both $\nabla_\tau$ and $\Delta_\tau$ annihilate the constants.
\end{proof}

Therefore, combining Proposition \ref{prop:b-conv-D-rep}, Theorem \ref{thm:A-T-inverse}, and (\ref{eqn:neumann}), we find the following representation of $\hat A (\Phi^+)$ as a finite difference operator on $D(\Phi^+)$.

\begin{proposition} \label{prop:A-D-rep}
For $p \in D(\Phi^+)$,
\begin{equation} \label{eqn:A-D-rep}
\hat A (\Phi^+) p = \sum_{k = 0}^{\lfloor d/2 \rfloor} \Big( - \frac{1}{2} \CD \Big)^k p.
\end{equation}
\end{proposition}

\subsection{An explicit algebraic formula for $\gog = \mathfrak{su}(n)$}
\label{sec:fd-inversion}

In this subsection we take $\gog = \mathfrak{su}(n)$.  We show that, given an assumption on the highest weights $(\lambda, \mu, \nu)$ that holds for ``typical'' triples, we can write $C_{\lambda \mu}^\nu$ explicitly as a linear combination of values of $\CJ(\lambda', \mu'; \gamma)$ at lattice points.  This result can be thought of as a partial inverse to the $\CJ$-LR relation (\ref{eqn:JLR1}).

Before stating the theorem, we need to give another definition.

\begin{definition} \label{def:shielded} \normalfont
Let $d = |\Phi^+| - r = \frac{1}{2}(n-1)(n-2).$ We will say that a compatible triple $(\lambda, \mu, \nu)$ of dominant weights of $\mathfrak{su}(n)$ is {\it shielded} if the points $\nu' + \lfloor d/2 \rfloor w(\rho)$, $w \in W$ are dominant and all lie in the interior of a single polynomial domain of $\CJ(\lambda', \mu' ; \gamma)$. \remdone
\end{definition}

\begin{remark} \label{rem:mostly-shielded} \normalfont
For $n = 2$, all compatible triples are shielded. For $n >2$ there are infinitely many non-shielded triples, but shielded triples are ``typical'' in the following sense.  The non-analyticities of $\CJ$ are contained within a centered hyperplane arrangement in $\tot^3$, and any compatible triple $(\lambda, \mu, \nu)$ such that $(\lambda', \mu', \nu')$ lies further than a distance $\lfloor d/2 \rfloor |\rho|$ from each of these hyperplanes is shielded. (In fact this condition is much stronger and excludes many shielded triples.) In particular, as $\lambda$ and $\mu$ both grow large, the ratio $$\frac{\#\{\, \nu \ | \ C_{\lambda \mu}^\nu \ne 0,\ (\lambda, \mu, \nu) \textrm{ shielded} \, \}}{\# \{\, \nu \ | \ C_{\lambda \mu}^\nu \ne 0 \,\}}$$ goes to 1. \remdone
\end{remark}

The main result of this subsection is then:

\begin{theorem} \label{thm:fin-dif-inversion}
For $(\lambda, \mu, \nu)$ a shielded triple of dominant weights of $\mathfrak{su}(n)$,
\begin{equation} \label{eqn:fin-dif-inversion}
C_{\lambda \mu}^\nu = \sum_{k = 0}^{\lfloor d/2 \rfloor} \Big( - \frac{1}{2} \CD \Big)^k \CJ(\lambda', \mu' ; \nu').
\end{equation}
\end{theorem}

\begin{proof}
The polynomial domains of $\CJ$ in $\tot^3$ are cones with apex at the origin.  This implies that for $(\lambda, \mu, \nu)$ shielded, there exists a polynomial $q$ on $\tot^3$ and an open cone $R$ with apex $-(\rho, \rho, \rho) \in \tot^3$, such that $\CJ(\alpha', \beta'; \gamma') = q(\alpha, \beta, \gamma)$ for all $(\alpha, \beta, \gamma) \in R$, and also $\CJ(\lambda', \mu'; \nu' + \tau) = q(\lambda, \mu, \nu + \tau)$ for all $\tau \in Q$ lying in the convex hull of the Weyl orbit of $\lfloor d/2 \rfloor \rho$.

Since all points $(\alpha', \beta', \gamma')$ with $(\alpha, \beta, \gamma) \in R$ lie on the interior of the same polynomial domain of $\CJ$, Proposition \ref{prop:inversion-reworked} yields a polynomial $p(\alpha, \beta, \gamma) := \hat A(\Phi^+)q(\alpha, \beta, \gamma)$ such that $C_{\eta \xi}^\theta = p(\eta, \xi, \theta)$ for all compatible triples $(\eta, \xi, \theta) \in R$.  Moreover, by (\ref{eqn:JLR-laplacian}), at any such point we have $$q(\eta, \xi, \theta) = \Big(1 + \frac{1}{2} \CD \Big) \mathscr{C}(\eta, \xi, \theta),$$ and if $(\eta, \xi, \theta)$ lies farther than a distance $|\rho|$ from the boundary of $R$, then $$\Big(1 + \frac{1}{2} \CD \Big) \mathscr{C}(\eta, \xi, \theta) = \Big(1 + \frac{1}{2} \CD \Big) p(\eta, \xi, \theta).$$
Therefore we must have the equality of polynomials $q = (1 + \frac{1}{2} \CD )p$, as this holds at all compatible triples in $R$ that lie sufficiently far from the boundary.

By Lemma \ref{lem:D-deg}, it is impossible that $\CD f = -2 f$ for a nonzero polynomial $f$, which means that $1 + \frac{1}{2} \CD$ is injective on polynomials.  Also by Lemma \ref{lem:D-deg}, $\CD^{\lfloor \deg f /2 \rfloor + 1} f = 0$, so that the Neumann series (\ref{eqn:neumann}) truncates, and we find: $$C_{\lambda \mu}^\nu = p(\lambda, \mu, \nu) = \sum_{k = 0}^{\lfloor d /2 \rfloor} \Big( -\frac{1}{2} \CD \Big)^k q(\lambda, \mu, \nu).$$
Each term $( -\frac{1}{2} \CD )^k q(\lambda, \mu, \nu)$ in the sum above is a linear combination of the values $q(\lambda, \mu, \nu + \tau)$, where $\tau \in Q$ lies in the convex hull of the Weyl orbit of $k \rho$.  Since $q(\lambda, \mu, \nu + \tau) = \CJ(\lambda', \mu' ; \nu' + \tau)$ for all such $\tau$, this gives the desired result (\ref{eqn:fin-dif-inversion}).
\end{proof}

\begin{remark} \label{rem:poly-vs-quasipoly} \normalfont
The polynomiality property of $C_{\lambda \mu}^\nu$ for $\mathfrak{su}(n)$, discussed in section \ref{sec:polynomiality}, is crucial in the proof above.  For arbitrary $\gog$ we know that $C_{\lambda \mu}^\nu$ is expressed by a piecewise quasi-polynomial function, but the Neumann series (\ref{eqn:neumann}) for $1 + \frac{1}{2} \CD$ may not converge when applied to quasi-polynomials rather than genuine polynomials.  This is why we need to take $\mathfrak{g} = \mathfrak{su}(n)$ in Theorem \ref{thm:fin-dif-inversion}.

The same concerns about convergence motivate the notion of a shielded triple.  Since (\ref{eqn:neumann}) may also fail to converge when applied to piecewise polynomials, Definition \ref{def:shielded} is designed to ensure that we only ever need to apply $\CD$ to a single local polynomial expression for $\CJ$. \remdone
\end{remark}

\begin{example} \label{ex:disc-invert} \normalfont
Comparing with Example \ref{ex:JLR-laplacian}, we find that for $(\lambda, \mu, \nu)$ a shielded triple of $\mathfrak{su}(4)$,
\begin{equation} \label{eqn:shielded-su4}
C_{\lambda \mu}^\nu = \left(1 - \frac{1}{24} \sum_{\alpha \in \Phi^+} \nabla_\alpha \Delta_\alpha \right) \CJ(\lambda', \mu', \nu').
\end{equation}
Note the similarity to equation (\ref{Ahat-su4}) computing $C_{\lambda \mu}^\nu$ from $\CJ$ using $\hat A(\Phi^+)$ rather than finite difference operators.  In light of Proposition \ref{prop:A-D-rep}, this comes as no surprise.

Similarly, for $(\lambda, \mu, \nu)$ a shielded triple of $\mathfrak{su}(5)$, we have:
\begin{equation} \label{eqn:shielded-su5}
C_{\lambda \mu}^\nu = \sum_{k = 0}^3 \Bigg[ - \frac{1}{30} \sum_{\alpha \in \Phi^+} \Bigg( \nabla_\alpha \Delta_\alpha + \frac{1}{12} \sum_{\substack{\beta \in \Phi^+ \\ \langle \beta, \alpha \rangle = 0}} \big ( \nabla_{\alpha + \beta} \Delta_{\alpha+\beta} + \nabla_{\alpha - \beta} \Delta_{\alpha-\beta} \big ) \Bigg) \Bigg]^k \CJ(\lambda', \mu', \nu').
\end{equation}
\remdone
\end{example}
\pagebreak

\section{Weight multiplicities}

We conclude this paper by showing how the ideas of the preceding sections can also be used to compute the weight multiplicities of irreducible representations of $\gog$, leading in particular to a formula for Kostka numbers that is analogous to Theorem \ref{thm:fin-dif-inversion}.  The theory for weight multiplicities is simpler than for tensor product multiplicities, so we only sketch the proofs, as they amount to simplified versions of arguments that we have already given above.

From the form (\ref{eqn:C-from-mult}) of the Kostant--Steinberg formula we find that for sufficiently large $k \in \N$,
\begin{equation} \label{eqn:mult-from-C}
\mult_\lambda(\mu) = C_{\lambda \, (k\rho)}^{\mu + k\rho}.
\end{equation}
Thus one can think of the weight multiplicities as degenerations of the tensor product multiplicities that depend only on two parameters rather than three.  All of the constructions that we have developed for tensor product multiplicities have analogues in this setting.

First we define the analogue of the volume function,
\begin{equation} \label{eqn:K-def}
\mathcal{I}(\alpha ; \beta) := \frac{\Delta_\gog(\alpha)}{\Delta_\gog(\rho)} \mathscr{F}^{-1} \big[ \mathcal{H}(i \, \cdot, \alpha) \big](\beta), \qquad \alpha, \beta \in \tot.
\end{equation}
As a function of $(\alpha, \beta) \in \tot^2$, $\mathcal{I}$ is a homogeneous piecewise polynomial of degree $d = |\Phi^+| - r$. We will usually fix $\alpha$ and consider $\mathcal{I}(\alpha ; \beta)$ as a function of $\beta$.  When $\Delta_\gog(\alpha) \not = 0$, $\mathcal{I}(\alpha ; \beta)$ is the density of the Duistermaat--Heckman measure for the action of the maximal torus on the coadjoint orbit $\mathcal{O}_\alpha$.  Equivalently, $(\Delta_\gog(\rho) / \Delta_\gog(\alpha)) \mathcal{I}(\alpha ; \beta)$ is the probability density for the orthogonal projection onto $\tot$ of a uniform random element of $\mathcal{O}_\alpha$.  When $\gog = \mathfrak{su}(n)$, this is the joint probability density of the diagonal entries of a uniform random traceless Hermitian matrix with eigenvalues $(\alpha_1, \hdots, \alpha_n)$; see \cite[sect.~2.3]{CMZ2}.

Just like tensor product multiplicities, each weight multiplicity $\mult_\lambda(\mu)$ equals the number of integer points in a certain polytope \cite{BGR, TB}, and it can be shown by the method of \cite[prop.~3]{CMZ} that the $d$-dimensional volume of this polytope is equal to $\mathcal{I}(\lambda ; \mu)$.

Taking Fourier transforms on both sides of the Kirillov character formula (\ref{eqn:KCF}), we obtain an analogue of Proposition \ref{prop:J-BS}:
\begin{equation} \label{eqn:conv-weights}
\mathcal{I}(\lambda' ; \beta) = b(\beta) \, * \! \sum_{\mu \in \lambda + Q} \mult_\lambda(\mu) \, \delta_\mu.
\end{equation}
This formula appeared in the literature at least as early as \cite[eqn.~5.3]{DRW}.  Restricting to the shifted root lattice $\lambda + Q$, we obtain a discrete version,
\begin{equation} \label{eqn:K-disc-conv}
\sum_{\mu \in \lambda + Q} \mathcal{I}(\lambda' ; \mu) \, \delta_{\mu} =  \Bigg (\sum_{\tau \in Q} b(\tau) \, \delta_\tau \Bigg) * \sum_{\tau \in \lambda + Q} \mult_\lambda(\tau) \, \delta_{\tau},
\end{equation}
and by the same Fourier-analytic method as in Theorem \ref{thm:C-from-J-Fourier} we find:
\begin{theorem} \label{thm:mult-from-K-Fourier}
For any dominant weights $\lambda, \mu$ of $\gog$,
\begin{equation} \label{eqn:mult-from-K-Fourier}
\mult_\lambda(\mu) = \frac{1}{(2\pi)^r|Q^\vee|} \int_{\tot / 2\pi Q^\vee} \frac{1}{R(x)} \sum_{\tau \in \lambda + Q} \mathcal{I}(\lambda' ; \tau) \, e^{i \langle \tau - \mu, x \rangle} \, dx.
\end{equation}
\end{theorem}
It also follows that $\mult_\lambda(\mu)$ can be computed algebraically from $\mathcal{I}$ using the algorithm of Lemma \ref{lem:disc-deconv}.

We now take $\gog = \mathfrak{su}(n)$. In this case the weight multiplicities are usually called Kostka numbers and we write them as $K_{\lambda}^{\mu}$ rather than $\mult_\lambda(\mu)$. From Theorem \ref{thm:bs-deconv} we have:

\begin{corollary} \label{cor:weight-inversion}
Let $\gog = \mathfrak{su}(n)$ and choose $\eta$ as in Theorem \ref{thm:bs-deconv}. For dominant weights $\lambda, \mu$ with $\mu \in \lambda + Q$,
\begin{equation} \label{eqn:inversion-formula-weights} K_{\lambda}^{\mu} = \lim_{t \to 0^+} \hat A(\Phi^+) \mathcal{I}(\lambda' ; \mu + t \eta),\end{equation} where the operator $\hat A(\Phi^+)$ acts in the second argument of $\mathcal{I}$.
\end{corollary}

By \cite[thm.~5.1]{BGR}, there is a piecewise polynomial function $q$ on $\tot^2$ of degree $d = \frac{1}{2}(n-1)(n-2)$ such that $K_\lambda^\mu = q(\lambda, \mu)$.  Moreover, \cite[thm.~3.2]{BGR} shows that the domains of polynomiality of $q$ are the same as those of $\mathcal{I}$.  This leads to the following analogue of Definition \ref{def:shielded}:

\begin{definition} \label{def:shielded-pair} \normalfont
We will say that a pair of weights $(\lambda, \mu)$ of $\mathfrak{su}(n)$ is {\it shielded} if $\lambda$ is dominant, $\mu \in \lambda + Q$, and the points $\mu + \lfloor d/2 \rfloor w(\rho)$, $w \in W$ all lie in the interior of a single polynomial domain of $\mathcal{I}(\lambda' ; \beta)$.
\remdone
\end{definition}

Then the same technique used to prove Theorem \ref{thm:fin-dif-inversion} gives:

\begin{theorem} \label{thm:fin-dif-inv-weights}
For $(\lambda, \mu)$ a shielded pair of weights of $\mathfrak{su}(n)$,
\begin{equation} \label{eqn:fin-dif-inv-weights}
K_\lambda^\mu = \sum_{k = 0}^{\lfloor d/2 \rfloor} \Big( - \frac{1}{2} \CD \Big)^k \mathcal{I}(\lambda'; \mu).
\end{equation}
\end{theorem}
\bigskip

\section*{Acknowledgements}
The author thanks Jean-Bernard Zuber and Govind Menon for their mentorship, as well as Robert Coquereaux, Mich\`ele Vergne, Mason Biamonte and Alex Moll for helpful discussions.  This research was partially supported by the National Science Foundation under Grant No.~DMS 1714187 and by the Chateaubriand Fellowship of the Embassy of France in the United States.
\pagebreak

\bibliographystyle{alpha}

\end{document}